\documentclass[a4paper,12pt,reqno]{amsart}

\usepackage{amsmath}
\usepackage{amsthm}
\usepackage{amssymb}
\usepackage{fullpage}

\newcommand{\N}{\mathbb{N}}
\newcommand{\Q}{\mathbb{Q}}
\newcommand{\R}{\mathbb{R}}
\newcommand{\C}{\mathbb{C}}

\renewcommand{\Re}{\operatorname{Re}}
\renewcommand{\Im}{\operatorname{Im}}

\newcommand{\mc}{\mathcal}
\newcommand{\mb}{\mathbf}
\newcommand{\rg}{\operatorname{rg}}

\newtheorem{lemma}{Lemma}[section]
\newtheorem{proposition}[lemma]{Proposition}
\newtheorem{theorem}[lemma]{Theorem}
\newtheorem{corollary}[lemma]{Corollary}
\theoremstyle{remark}
\newtheorem{remark}[lemma]{Remark}

\theoremstyle{definition}
\newtheorem{definition}[lemma]{Definition}

\numberwithin{equation}{section}

\DeclareMathOperator{\artanh}{artanh}

\title{Strichartz estimates for the one-dimensional wave equation}

\author{Roland Donninger}
\address{Universit\"at Wien, Fakult\"at f\"ur Mathematik,
  Oskar-Morgenstern-Platz 1, 1090 Vienna, Austria}
\email{roland.donninger@univie.ac.at}

\author{Irfan Glogi\'c}
\address{Universit\"at Wien, Fakult\"at f\"ur Mathematik,
  Oskar-Morgenstern-Platz 1, 1090 Vienna, Austria}
\email{irfan.glogic@univie.ac.at}

\thanks{Both authors acknowledge support by the Austrian Science Fund
  FWF, Project P 30076, ``Self-similar blowup in dispersive wave equations''.}

\begin{document}

\begin{abstract}
We study the hyperboloidal initial value problem for the
one-dimensional wave equation perturbed by a smooth potential. We
show that the evolution decomposes into a finite-dimensional spectral
part and an
infinite-dimensional radiation part. For the radiation part we prove a set of Strichartz
estimates. As an application we study the long-time asymptotics of
Yang-Mills fields on a wormhole spacetime.
  \end{abstract}

\maketitle

\section{Introduction}
\noindent Strichartz estimates were originally 
discovered in the context of the Fourier restriction problem
\cite{Str77} but only later their true power was exploited 
in the study of nonlinear wave equations \cite{LinSog95}. To illustrate
this point, consider for instance the Cauchy problem for the cubic
wave equation in three spatial dimensions,
\begin{equation}
  \label{eq:cubic}
  \begin{cases}
    (\partial_t^2-\Delta_x)u(t,x)=u(t,x)^3 & (t,x)\in \R\times\R^3\\
    u(t,x)=f(x),\quad \partial_t u(t,x)=g(x) & (t,x)\in \{0\}\times \R^3,
  \end{cases}
  \end{equation}
for given initial data $f,g\in \mc S(\R^3)$, say. A weak formulation of
Eq.~\eqref{eq:cubic} is provided by Duhamel's formula
\begin{equation}
  \label{eq:cubicweak}
  u(t,\cdot)=\cos(t|\nabla|)f+\frac{\sin(t|\nabla|)}{|\nabla|}g
+\int_0^t \frac{\sin((t-t')|\nabla|)}{|\nabla|}\left
  (u(t',\cdot)^3\right)dt'
\end{equation}
with the \emph{wave propagators} $\cos(t|\nabla|)$ and
$\frac{\sin(t|\nabla|)}{|\nabla|}$. The latter are the Fourier multipliers
that yield the solution to the
free wave equation $(\partial_t^2 -\Delta_x)u(t,x)=0$.
The point is that Eq.~\eqref{eq:cubicweak} is a reformulation of
Eq.~\eqref{eq:cubic} as a fixed point problem. Proving the existence
of solutions 
to Eq.~\eqref{eq:cubic} therefore amounts to showing that the
operator on the right-hand side of Eq.~\eqref{eq:cubicweak}
has a fixed point. The main issue then is to find suitable spaces that
are compatible with the free evolution and that allow one to control
the nonlinear term. For the cubic equation \eqref{eq:cubicweak} the
Sobolev embedding $\dot H^1(\R^3)\hookrightarrow L^6(\R^3)$ suffices
but if one increases the power of the nonlinearity or the spatial dimension, a more
sophisticated argument is required. The crucial tool is provided by
the Strichartz estimates which are mixed spacetime bounds on the wave
propagators of the form
\[ \|\cos(t|\nabla|)f\|_{L^p_t(\R)L^q(\R^d)}:=\left (\int_\R \left
      \|\cos(t|\nabla|)f\right \|_{L^q(\R^d)}^pdt\right )^{1/p}\lesssim
  \|f\|_{\dot H^s(\R^d)} \]
for certain admissible values of $p$, $q$, $s$, and $d$.
For instance, the sine propagator satisfies
the Strichartz estimate
\[ \left
\|\frac{\sin(t|\nabla|)}{|\nabla|}g\right\|_{L^5_t(\R)L^{10}(\R^3)}\lesssim
\|g\|_{L^2(\R^3)} \]
which allows one to control a quintic nonlinearity in three dimensions. At the same time,
the Strichartz estimates provide information on the long-time
asymptotics which makes them crucial in proving scattering.

The physical effect that is responsible for the existence of Strichartz
estimates is dispersion. The latter refers to the observation that
waves of different frequencies travel at different speeds. In other
words, a wave packet tends to spread out which leads to
an averaged decay that is quantified by the Strichartz
estimates. The strength of the dispersive decay depends strongly on
the underlying spatial dimension: The higher the space dimension, the more room
there is for the wave to spread out. On the other hand, in the
one-dimensional case, there is no dispersion at all and the evolution
is a pure transport phenomenon.
This precludes the existence of Strichartz estimates as
 is easily seen by noting that
$u(t,x)=\chi(t-x)$ for a $\chi\in C^\infty_c(\R)$ solves
$(\partial_t^2-\partial_x^2)u(t,x)=0$. By translation invariance we
have
\[ \|u(t,\cdot)\|_{L^q(\R)}=\|\chi\|_{L^q(\R)} \]
and $\|u\|_{L^p(\R)L^q(\R)}=\infty$, unless $p=\infty$.
The weak dispersion in low dimensions causes severe
difficulties in understanding the asymptotics of many models in
quantum field theory, see e.g.~\cite{LinSof15, Ste16, KowMarMun17} for
recent work.

In this paper we show that one can recover Strichartz
estimates even in the one-dimensional case if one studies a
\emph{hyperboloidal evolution problem} instead of the standard Cauchy
problem.
The key observation is that the standard Cartesian coordinates are
not very well suited for describing radiation processes.
The foliation induced by the
standard coordinates is singular at null infinity and therefore
unnatural in this context, see e.g.~\cite{DonZen14} for a discussion
on this. Consequently, as
suggested in many physics papers, e.g.~\cite{Fri83, Zen08, Zen11, BizMac17}, we choose a hyperboloidal
foliation instead, where the leaves are asymptotic to translated
forward lightcones.
In this setup we study the evolution problem for the
one-dimensional wave equation with an arbitrary potential added (to
avoid technicalities we restrict ourselves to smooth potentials). We
show that the solution decomposes into a finite dimensional part
which is controlled by spectral theory and an infinite-dimensional
``radiation'' part which satisfies Strichartz estimates, provided a
certain spectral assumption holds. We remark in passing that there are
some technical similarities with Strichartz estimates in the context
of self-similar blowup established in \cite{Don17, DonRao18}.

As a first application we consider Yang-Mills fields on a wormhole
geometry. Under a certain symmetry reduction, we study small-energy
perturbations of an
explicit Yang-Mills connection and prove its asymptotic stability in a
Strichartz sense.

\subsection{Main results}
We use the hyperboloidal coordinates from \cite{BizMac17} defined by 
$\Phi: \R\times (-1,1)\to \R^2$,
\[ \Phi(s,y):=(s-\log\sqrt{1-y^2}, \artanh y). \]
The map $\Phi$ is a diffeomorphism onto its image with inverse
\[ \Phi^{-1}(t,x)=(t-\log\cosh x,\tanh x). \]
In these coordinates, the one-dimensional wave equation
\begin{equation}
  \label{eq:1dwave}
  (\partial_t^2-\partial_x^2)v(t,x)=0
\end{equation}
reads
\begin{equation}
  \label{eq:1dwavesy}
    \left
      [\partial_s^2+2y\partial_s\partial_y+\partial_s-(1-y^2)\partial_y^2+2y\partial_y\right]u(s,y)=0,
    \end{equation}
    where $v(t,x)=u(t-\log\cosh x, \tanh x)$.
    By testing with $\partial_s u(s,y)$, we formally find the energy identity
    \begin{equation}
      \label{eq:enid}\frac12\frac{d}{ds}\left [\int_{-1}^1 (1-y^2)|\partial_y u(s,y)|^2 dy+\int_{-1}^1
      |\partial_s u(s,y)|^2 dy\right ] =-|\partial_s
    u(s,-1)|^2-|\partial_s u(s,1)|^2
    \end{equation}
  and this motivates the introduction of the following \emph{energy norm}.
  \begin{definition}
    For functions $(f,g)\in C^1(-1,1)\times C(-1,1)$, we define the
    \emph{energy norm} $\|(f,g)\|_{\mc H}$ by
    \[ \|(f,g)\|_{\mc H}^2:=\int_{-1}^1 (1-y^2)|f'(y)|^2 dy+\int_{-1}^1
      |g(y)|^2 dy. \] 
  \end{definition}
Our main result is concerned with a more general class of wave equations,
that is to say, we study the initial
value problem
\begin{equation}
  \label{eq:HIVVintro}
  \begin{cases}
    [\partial_s^2+2y\partial_s\partial_y+\partial_s-(1-y^2)\partial_y^2+2y\partial_y+V(y)]u(s,y)=0
  & (s,y)\in (0,\infty)\times (-1,1)
  \\
  u(s,y)=f(y),\qquad \partial_s u(s,y)=g(y) & (s,y)\in \{0\}\times (-1,1)
  \end{cases}
\end{equation}
for an unknown $u: [0,\infty)\times (-1,1)\to\C$, prescribed \emph{initial data}
$f,g: (-1,1)\to\C$, and a given \emph{potential} $V:
(-1,1)\to\C$. 

\begin{definition}
  \label{def:SigmaV}
  Let $V\in C^\infty([-1,1])$ be even. We define a set $\Sigma_V\subset\C$ by saying
  that $\lambda\in \C$ belongs to $\Sigma_V$ if $\Re\lambda\geq 0$ and
  there exists a nontrivial odd $f_\lambda \in C^\infty([-1,1])$
  that satisfies
  \[
    -(1-y^2)f_\lambda''(y)+2(\lambda+1)yf_\lambda'(y)+\lambda(\lambda+1)f_\lambda(y)+V(y)f_\lambda(y)=0 \]
  for all $y\in (-1,1)$.
\end{definition}

\begin{theorem}
  \label{thm:main}
  Let $V: [-1,1]\to\C$ be smooth and even, $p\in [2,\infty]$, $q\in
  [1,\infty)$, and $\epsilon>0$. Then there exist
  constants $C_{p,q}, C_\epsilon>0$ such that the following holds.
  \begin{enumerate}
  \item The set $\Sigma_V^+:=\Sigma_V \cap \{z\in \C: \Re z>0\}$ consists of
    finitely many points.
  \item For any given odd initial data $f,g\in C^\infty([-1,1])$, there exists a
      unique solution $u=u_{f,g}\in C^\infty([0,\infty)\times (-1,1))$ to
      the initial value problem \eqref{eq:HIVVintro} that satisfies,
      for each $s\geq 0$,
      \[ \|(u(s,\cdot),\partial_s u(s,\cdot))\|_{\mc H}<\infty. \]
    \item 
      For each $\lambda\in \Sigma_V^+$ there exists a number $n(\lambda)\in
      \N_0$ and
      a set $\{\phi^{\lambda,k}_{f,g}\in C^\infty(-1,1): k\in \{0,1,\dots,n(\lambda)\}\}$
      of odd functions satisfying
      $\|(\phi_{f,g}^{\lambda,k},0)\|_{\mc
        H}<\infty$ and such that the solution $u_{f,g}$ decomposes according to
      \[ u_{f,g}(s,y)=\sum_{\lambda\in \Sigma_V^+}
        e^{\lambda s}\sum_{k=0}^{n(\lambda)}s^k \phi_{f,g}^{\lambda,k}(y)+\widetilde
        u_{f,g}(s,y). \]
The map $(f,g)\mapsto \phi_{f,g}^{\lambda,k}$ has finite rank
      and 
      \begin{align*}
        \left \|(\widetilde u_{f,g}(s,\cdot),\partial_s \widetilde
          u_{f,g}(s,\cdot))\right \|_{\mc H}&\leq C_\epsilon 
        e^{\epsilon s}\|(f,g)\|_{\mc H} \end{align*}
      for all $s\geq 0$.
      \item If $\Sigma_V\cap i\R=\emptyset$, we have the Strichartz
        estimates
        \[ \|\widetilde u_{f,g}\|_{L^p(0,\infty)L^q(-1,1)}\leq
          C_{p,q}\|(f,g)\|_{\mc H}. \]
  \end{enumerate}
\end{theorem}

\begin{remark}
  \label{rem:impren}
  With slightly more effort it is also possible to improve the energy
  bound to
  \[ \|(\widetilde u_{f,g}(s,\cdot), \partial_s \widetilde
    u_{f,g}(s,\cdot))\|_{\mc H}\lesssim \|(f,g)\|_{\mc H} \]
  for all $s\geq 0$,
  provided $\Sigma_V \cap i\R=\emptyset$. To keep the paper at a
  reasonable length, however, we refrain from working out the details.
\end{remark}

\begin{remark}
The smoothness assumptions are imposed for convenience and can of
course be considerably weakened. This produces some inessential
technicalities but no new insight. 
  \end{remark}

\section{Application: Asymptotics of Yang-Mills fields on wormholes}

\noindent We give an application of Theorem \ref{thm:main} to
Yang-Mills fields on wormholes studied in \cite{BizMac17}.

\subsection{Setup}
 As in \cite{BizMac17}, we consider
$M^4:=\R\times \R\times (0,\pi)\times (0,2\pi)$.
Let $(t,r,\theta,\varphi): M^4\to \R^4$ be a chart on $M^4$.
We define a Lorentzian metric $g$ on $M^4$ by
\[ g:=-dt\otimes dt+dr\otimes dr+\cosh(r)^2 (d\theta\otimes d\theta +
  \sin(\theta)^2 d\varphi\otimes d\varphi). \]
Then $(M^4,g)$ is a Lorentzian manifold with 2 asymptotic ends (as
$r\to\pm \infty$), which physically represents a wormhole spacetime.
We would like to study Yang-Mills connections on the principal bundle
$M^4\times \mathrm{SU}(2)$.
That is to say, we are looking for $\mathrm{su}(2)$-valued one-forms
\[ A=A_0 dt+A_1 dr+A_2 d\theta + A_3 d\varphi \]
on $M^4$ that formally\footnote{``Formally'' here means that we are in
  fact looking for solutions of the Euler-Lagrange equation associated
  to the Yang-Mills action.}
extremize the Yang-Mills action
\[ \int_{(M^4,g)} \mathrm{tr}(F_{\mu\nu}F^{\mu\nu}), \]
where $F_{\mu\nu}:=\partial_\mu A_\nu-\partial_\nu
A_\mu+[A_\mu,A_\nu]$ is the curvature two-form.
The Euler-Lagrange equation associated to the Yang-Mills action reads
\begin{equation}
  \label{eq:YM}
  \frac{1}{\sqrt{|\det g|}}\partial_\mu\left (\sqrt{|\det
      g|}F^{\mu\nu}\right )+[A_\mu,F^{\mu\nu}]=0
  \end{equation}
and is called the Yang-Mills equation.
Here, Greek indices run from $0$ to $3$ and we use Einstein's
summation convention. As usual, indices are raised and lowered by the
metric, i.e., $F^{\mu\nu}=g^{\mu\alpha}g^{\nu\beta}F_{\alpha\beta}$,
where $g_{\mu\nu}=g(\partial_\mu,\partial_\nu)$ with
$\partial_0=\partial_t$, $\partial_1=\partial_r$,
$\partial_2=\partial_\theta$, $\partial_3=\partial_\varphi$, and
$g^{\mu\nu}$ is defined by the requirement that
$g^{\mu\alpha}g_{\alpha\nu}=\delta^\mu{}_\nu$, where
$\delta^\mu{}_\nu$ is the Kronecker symbol. Furthermore, $\det g=-\cosh(r)^4\sin(\theta)^2$ is the
determinant of the matrix $(g_{\mu\nu})$. We choose the basis
$\{\tau_1,\tau_2,\tau_3\}$ for $\mathrm{su}(2)$, where
\[ \tau_1:=-\frac{i}{2}
  \begin{pmatrix}
    0 & 1\\
    1 & 0
  \end{pmatrix},
  \qquad
  \tau_2:=-\frac{i}{2}
  \begin{pmatrix}
    0 & -i \\
    i & 0
  \end{pmatrix},
  \qquad
  \tau_3:=-\frac{i}{2}
  \begin{pmatrix}
    1 & 0 \\
    0 & -1
  \end{pmatrix}.
\]
Then $\cos\theta \tau_3 d\varphi$ solves the Yang-Mills equation, as
is easily checked.
We would like to study the stability of the explicit solution
$\cos\theta\tau_3 d\varphi$. Following \cite{BizMac17}, we consider the perturbation
ansatz
\[ A=\cos\theta\tau_3d\varphi+W(t,r)(\tau_1 d\theta+\sin\theta \tau_2
  d\varphi) \]
for a real-valued function $W$.
By noting the commutator relations $[\tau_1,\tau_2]=\tau_3$, 
$[\tau_1,\tau_3]=-\tau_2$, and $[\tau_2,\tau_3]=\tau_1$,
we readily
compute the nonvanishing components of $F_{\mu\nu}$, which are given by
$F_{02}=\partial_0 W\tau_1$, $F_{03}=\partial_0
  W\sin\theta\tau_2$, 
  $F_{12}=\partial_1 W\tau_1$,
  $F_{13}=\partial_1 W\sin\theta\tau_2$,
  and $F_{23}=-(1-W^2)\sin\theta\tau_3$.
  Consequently,
  \begin{align*}
    F^{02}&=-\frac{\partial_0 W\sin\theta}{\sqrt{|\det g|}}\tau_1,
  &
      F^{03}&=-\frac{\partial_0 W}{\sqrt{|\det g|}}\tau_2,
    &
      F^{12}&=\frac{\partial_1 W\sin\theta}{\sqrt{|\det g|}}\tau_1, \\
    F^{13}&=\frac{\partial_1 W}{\sqrt{|\det g|}}\tau_2,
    &
     F^{23}&=-\frac{1-W^2}{\cosh(r)^2 \sqrt{|\det g|}}\tau_3,
  \end{align*}
  and for $\nu\in \{2,3\}$, Eq.~\eqref{eq:YM} reduces to
\begin{equation}
  \label{eq:YMW}
  (\partial_t^2-\partial_r^2)W(t,r)=\frac{W(t,r)(1-W(t,r)^2)}{\cosh(r)^2},
\end{equation}
whereas for $\nu\in \{0,1\}$, Eq.~\eqref{eq:YM} is satisfied identically.
In particular, we observe that $W=0$ is a
solution, showing that $\cos\theta\tau_3d\varphi$ indeed solves
the Yang-Mills equation.
Consequently, under this particular symmetry reduction enforced by the
perturbation ansatz,
the study of the stability of $\cos\theta\tau_3 d\varphi$ as a
solution to the Yang-Mills equation boils down to analyzing the
stability of the trivial solution $W=0$ of Eq.~\eqref{eq:YMW}. Note that
Eq.~\eqref{eq:YMW} is effectively a one-dimensional semilinear wave
equation and studying its asymptotics might seem hard due to the lack
of dispersion.

Evidently, there are two more trivial solutions to Eq.~\eqref{eq:YMW},
namely $W=\pm 1$. In \cite{BizMac17} it is shown that these
solutions are linearly stable under general perturbations, whereas the
solution $W=0$ has one unstable mode. In this paper, we restrict
ourselves to odd solutions. Under this assumption, the solutions
$W=\pm 1$ disappear and $W=0$ becomes stable, as we will see. The restriction to
odd solutions is important for our method of proof, where the free
wave equation and the corresponding explicit solution formula is taken
as a starting point. We believe that our results remain valid in the
general case but this would require a different proof. We plan to
address this issue in the future. 

\subsection{Hyperboloidal formulation}

The hyperboloidal initial value problem for the Yang-Mills equation
\eqref{eq:YMW} takes the
form
\begin{equation}
  \label{eq:YMhyp}
    \begin{cases}
    [\partial_s^2+2y\partial_s\partial_y+\partial_s-(1-y^2)\partial_y^2+2y\partial_y-1]u(s,y)=-u(s,y)^3
  & (s,y)\in (0,\infty)\times (-1,1)
  \\
  u(s,y)=f(y),\qquad \partial_s u(s,y)=g(y) & (s,y)\in \{0\}\times (-1,1)
  \end{cases},
\end{equation}
where $W(t,r)=u(t-\log\cosh r,\tanh r)$.
Note that the linear part in \eqref{eq:YMhyp} is
Eq.~\eqref{eq:HIVVintro} with $V(y)=-1$. 
We compute $\Sigma_V$.
\begin{lemma}
  \label{lem:specYM}
  Let $V(y)=-1$ for all $y\in [-1,1]$. Then $\Sigma_V=\emptyset$.
\end{lemma}

\begin{proof}
  According to Definition \ref{def:SigmaV}, we have to solve the
  spectral problem
\[
  -(1-y^2)f''(y)+2(\lambda+1)yf'(y)+\lambda(\lambda+1)f(y)-f(y)=0 \]
for $f\in C^\infty([-1,1])$ odd and $\Re\lambda\geq 0$.
This ODE can be solved explicitly in terms of hypergeometric functions
and it follows that no solution other than $f=0$ exists. We refrain
from giving the details here because this spectral problem was discussed at
length in
\cite{BizMac17}.
\end{proof}

\begin{definition}
  Set $V(y):=-1$ for all $y\in [-1,1]$ and let $u_{f,g}$ be the
  solution of Eq.~\eqref{eq:HIVVintro} provided by Theorem \ref{thm:main}.
  We define the \emph{wave propagators} by
  \[ C(s)f:=u_{f,0}(s,\cdot),\qquad S(s)g:=u_{0,g}(s,\cdot). \]
\end{definition}

By Theorem \ref{thm:main} and Lemma \ref{lem:specYM} we have the bound $\|S(s)
g\|_{L^6(-1,1)}\lesssim \|g\|_{L^2(-1,1)}$ for any $s\geq 0$ and thus,
$S(s)$ uniquely extends to a bounded operator $S(s):
L^2_\mathrm{odd}(-1,1)\to L_\mathrm{odd}^6(-1,1)$, where
$L^q_\mathrm{odd}(-1,1)$ denotes the completion of
\[ \{f\in
C^\infty([-1,1]): f \mbox{ is odd}\} \] with
respect to $\|\cdot\|_{L^q(-1,1)}$.
Then a weak formulation of Eq.~\eqref{eq:YMhyp} is given by
\begin{equation}
  \label{eq:YMhypweak}
  u(s,\cdot)=C(s)f+S(s)g-\int_0^s S(s-s')\left (u(s',\cdot)^3\right )ds'.
\end{equation}

\begin{theorem}
  \label{thm:YM}
  There exists a $\delta>0$ such that for all odd functions
  $f,g\in C^\infty([-1,1])$ with $\|(f,g)\|_{\mc H}<\delta$, Eq.~\eqref{eq:YMhypweak} 
 has a unique
 solution $u$ in $C([0,\infty),L^6_\mathrm{odd}(-1,1))$. Furthermore,
 $u\in L^p((0,\infty),L^6_\mathrm{odd}(-1,1))$ for any $p\in [3,\infty]$.
\end{theorem}

\begin{remark}
  Theorem \ref{thm:YM} implies that the Yang-Mills connection
  $\cos\theta\tau_3 d\varphi$ is \emph{asymptotically stable} under
  odd small-energy perturbations on the hyperboloid
  \[ \left \{(-\tfrac12
      \log(1-y^2),\artanh y)\in \R^{1,1}: y\in (-1,1) \right \}. \]
\end{remark}

\subsection{Proof of Theorem \ref{thm:YM}}

  For $R>0$ we define
  \begin{align*} X_R:=
  \{&u\in C([0,\infty),L^6_{\mathrm{odd}}(-1,1)): 
    \|u\|_{L^3(0,\infty)L^6(-1,1)}
      +\|u\|_{L^\infty(0,\infty)L^6(-1,1)}\leq R\}.
 \end{align*}
Note that $u\in X_R$ implies $u(s,\cdot)\in L_\mathrm{odd}^6(-1,1)$ and thus,
$u(s,\cdot)^3\in L_\mathrm{odd}^2(-1,1)$ for any $s\geq 0$. Consequently, 
\[ \mc K_{f,g}(u)(s):=C(s)f+S(s)g-\int_0^s S(s-s')\left
    (u(s',\cdot)^3\right)ds' \]
is well-defined as a map $\mc K_{f,g}: X_R\to
C([0,\infty),L^6_\mathrm{odd}(-1,1))$ for any $R>0$ by Theorem \ref{thm:main}.

\begin{lemma}
  \label{lem:mapK}
  There exist $M,\delta_0>0$ such that for all $\delta\in
  (0,\delta_0]$ and any pair of odd functions
  $f,g\in C^\infty([-1,1])$ satisfying $\|(f,g)\|_{\mc H}<\delta$, $\mc
  K_{f,g}$ maps $X_{M\delta}$ to itself.
\end{lemma}

\begin{proof}
  Let $u\in X_R$ for some $R>0$ and assume that $\|(f,g)\|_{\mc H}<\delta$.
  By Theorem \ref{thm:main} and Lemma \ref{lem:specYM} we have
  \begin{align*}
  \|&\mc K_{f,g}(u)\|_{L^3(0,\infty)L^6(-1,1)} \\
  &\lesssim
                                               \|(f,g)\|_{\mc
    H}+\int_0^\infty \left \|1_{[0,s]}(s')S(s-s')\left
      (u(s',\cdot)^3\right)
  \right\|_{L^3_s(0,\infty)L^6(-1,1)}ds' \\
  &=\|(f,g)\|_{\mc
    H}+\int_0^\infty \left \|S(s-s')\left
      (u(s',\cdot)^3\right)
    \right\|_{L^3_s(s',\infty)L^6(-1,1)}ds' \\
  &=\|(f,g)\|_{\mc
    H}+\int_0^\infty \left \|S(s)\left
      (u(s',\cdot)^3\right)
    \right\|_{L^3_s(0,\infty)L^6(-1,1)}ds' \\
  &\lesssim \|(f,g)\|_{\mc H}+\int_0^\infty \|u(s',\cdot)^3\|_{L^2(-1,1)}ds'
  \\
  &=\|(f,g)\|_{\mc H}+\int_0^\infty \|u(s',\cdot)\|_{L^6(-1,1)}^3 ds'
  \\
  &=\|(f,g)\|_{\mc H}+\|u\|_{L^3(0,\infty)L^6(-1,1)}^3.
  \end{align*}
  Analogously,
    $\|\mc K_{f,g}(u)\|_{L^\infty(0,\infty)L^6(-1,1)}\lesssim
    \|(f,g)\|_{\mc H}+\|u\|_{L^3(0,\infty)L^6(-1,1)}^3$
  and we obtain
  \[
    \|\mc K_{f,g}(u)\|_{L^3(0,\infty)L^6(-1,1)}+\|\mc K_{f,g}(u)\|_{L^\infty(0,\infty)L^6(-1,1)}\leq
    C\delta+CR^3 \]
  for some constant $C>0$.
  Now we choose $\delta_0=(8C^3)^{-\frac12}$ and $R=2C\delta$.
  Then we have
  \[ C\delta+CR^3= C\delta+C(2C\delta)^3\leq C\delta+8C^3\delta_0^2
    C\delta=2C\delta \] and the claim follows with $M=2C$ since the wave propagators
  preserve oddness.
\end{proof}

Now we set up an iteration by $u_0:=0$ and
$u_n:=\mc K_{f,g}(u_{n-1})$ for $n\in \N$.
For brevity we define
\[
  \|u\|_X:=\|u\|_{L^3(0,\infty)L^6(-1,1)}+\|u\|_{L^\infty(0,\infty)L^6(-1,1)}. \]

\begin{lemma}
  \label{lem:Cauchy}
  There exist $M,\delta>0$ such that $u_n\in X_{M\delta}$ for all
  $n\in \N$ and the sequence $(u_n)_{n\in \N}$ is Cauchy with respect
  to $\|\cdot\|_X$, provided that $\|(f,g)\|_{\mc H}<\delta$.
\end{lemma}

\begin{proof}
  The first statement follows from Lemma \ref{lem:mapK}.
  The algebraic identity $a^3-b^3=(a-b)(a^2+ab+b^2)$ and H\"older's
  inequality yield
  \begin{align*}\|&u_{n+1}-u_n\|_{L^3(0,\infty)L^6(-1,1)} \\
    &= \|
    \mc K_{f,g}(u_{n})-\mc K_{f,g}(u_{n-1})\|_{L^3(0,\infty)L^6(-1,1)} \\
    &\lesssim \int_0^\infty \left \|S(s)\left
      [u_{n}(s',\cdot)^3-u_{n-1}(s',\cdot)^3\right ]\right
      \|_{L^3_s(0,\infty)L^6(-1,1)}ds' \\
    &\lesssim \int_0^\infty
      \|u_{n}(s',\cdot)^3-u_{n-1}(s',\cdot)^3\|_{L^2(-1,1)}ds' \\
    &\lesssim \int_0^\infty
      \|u_{n}(s',\cdot)-u_{n-1}(s',\cdot)\|_{L^6(-1,1)}
      \left
      (\|u_{n}(s',\cdot)\|_{L^6(-1,1)}^2+\|u_{n-1}(s',\cdot)\|_{L^6(-1,1)}^2
      \right )ds' \\
    &\lesssim \|u_{n}-u_{n-1}\|_{L^3(0,\infty)L^6(-1,1)}\left
      (\|u_{n}\|_{L^3(0,\infty)L^6(-1,1)}^2
      +\|u_{n-1}\|_{L^3(0,\infty)L^6(-1,1)}^2\right ).
  \end{align*}
  Analogously, we obtain the bound
  \begin{align*}\|
    &u_{n+1}-u_n\|_{L^\infty(0,\infty)L^6(-1,1)} \\
    &\lesssim \|u_{n}-u_{n-1}\|_{L^3(0,\infty)L^6(-1,1)}\left
      (\|u_{n}\|_{L^3(0,\infty)L^6(-1,1)}^2
    +\|u_{n-1}\|_{L^3(0,\infty)L^6(-1,1)}^2\right )
  \end{align*}
  and in summary,
 $\|u_{n+1}-u_n\|_X\leq CM^2\delta^2\|u_{n}-u_{n-1}\|_X$
  for some constant $C>0$. Thus, by choosing $\delta$ sufficiently
  small, we find $\|u_{n+1}-u_n\|_X\leq \frac12 \|u_{n}-u_{n-1}\|_X$
  for all $n\in \N$
  and this implies the claim.
\end{proof}

As a consequence of Lemma \ref{lem:Cauchy}, the sequence
$(u_n)_{n\in\N}$ converges to an element
\[ u\in C([0,\infty),L^6_\mathrm{odd}(-1,1))\cap
  L^3((0,\infty),L^6_\mathrm{odd}(-1,1))
  \cap L^\infty((0,\infty),L^6_\mathrm{odd}(-1,1)), \]
which satisfies Eq.~\eqref{eq:YMhypweak}.
It remains to prove the uniqueness.

\begin{lemma}
  \label{lem:YMunique}
  Let $f,g\in C^\infty([-1,1])$ be odd.
  Then there exists at most one function $u\in
  C([0,\infty),L^6_\mathrm{odd}(-1,1))$ that satisfies
  Eq.~\eqref{eq:YMhypweak}. 
\end{lemma}

\begin{proof}
  Suppose $u,\widetilde u\in C([0,\infty),L^6_\mathrm{odd}(-1,1))$
  satisfy Eq.~\eqref{eq:YMhypweak} and let $s_0>0$ be arbitrary. Then,
  for $s\in [0,s_0]$, we have
  \begin{align*}
    \|u(s,\cdot)-\widetilde u(s,\cdot)\|_{L^6(-1,1)}&\leq \int_0^s
    \left \|S(s-s')[u(s',\cdot)^3-\widetilde u(s',\cdot)^3]\right
    \|_{L^6(-1,1)}ds' \\
    &\lesssim \int_0^s \|u(s',\cdot)^3-\widetilde
      u(s',\cdot)^3\|_{L^2(-1,1)}ds' \\
    &\lesssim \left (\|u\|_{L^\infty(0,s_0)L^6(-1,1)}^2+\|\widetilde
      u\|_{L^\infty(0,s_0)L^6(-1,1)}^2
      \right ) \\
    &\quad \times \int_0^s \|u(s',\cdot)-\widetilde u(s',\cdot)\|_{L^6(-1,1)}ds'
  \end{align*}
and Gronwall's inequality implies that
$\|u(s,\cdot)-\widetilde u(s,\cdot)\|_{L^6(-1,1)}=0$ for all $s\in [0,s_0]$.
\end{proof}

\section{The hyperboloidal initial value problem for the free wave equation}

\noindent
Now we turn to the proof of Theorem \ref{thm:main} and
start with the hyperboloidal initial value problem for the free
wave equation, i.e., we study
\begin{equation}
  \label{eq:HIV}
  \begin{cases}
    [\partial_s^2+2y\partial_s\partial_y+\partial_s-(1-y^2)\partial_y^2+2y\partial_y]u(s,y)=0
  & (s,y)\in (0,\infty)\times (-1,1)
  \\
  u(s,y)=f(y),\qquad \partial_s u(s,y)=g(y) & (s,y)\in \{0\}\times (-1,1)
  \end{cases}
\end{equation}
for an unknown $u: [0,\infty)\times (-1,1)\to \R$ and given data $f,g: (-1,1)\to\R$.

\subsection{Classical solution of the initial value problem}
The solution to \eqref{eq:HIV} can be given explicitly. This is a
straightforward consequence of 
the fact that the general solution of Eq.~\eqref{eq:1dwave} is of
the form $v(t,x)=F(t-x)+G(t+x)$.

\begin{definition}
    For $f,g\in C^\infty(-1,1)$ and $(s,y)\in
    [0,\infty)\times (-1,1)$ we set
 \begin{align*}
u_{f,g}(s,y)&:=f(0)-\frac12\int_{-1+e^{-s}(1+y)}^0 (1-x)f'(x)dx
                                +\frac12\int_0^{1-e^{-s}(1-y)}(1+x)f'(x)dx  \\
    &\quad +\frac12\int_{-1+e^{-s}(1+y)}^{1-e^{-s}(1-y)}g(x)dx.
    \end{align*}
  \end{definition}

\begin{lemma}[Existence and uniqueness of smooth solutions]
  \label{lem:dAlembert}
  Let $f,g\in C^\infty(-1,1)$. Then
  $u_{f,g}\in C^\infty([0,\infty)\times (-1,1))$ and $u=u_{f,g}$
  is a solution to \eqref{eq:HIV}. Furthermore, this solution is
  unique in $C^\infty([0,\infty)\times (-1,1))$.
  \end{lemma}

  \begin{proof}
    Since
    $-1+e^{-s}(1+y)\in (-1,1)$ and
    $1-e^{-s}(1-y)\in (-1,1)$
    for all $s\geq 0$ and $y\in (-1,1)$,
it is evident that $u_{f,g}\in C^\infty([0,\infty)\times (-1,1))$ and a
straightforward computation shows that $u=u_{f,g}$ solves
\eqref{eq:HIV}.
In fact, the formula for $u_{f,g}$ is derived from the \emph{general}
solution $v(t,r)=F(t-r)+G(t+r)$ of Eq.~\eqref{eq:1dwave} and
thus, $u_{f,g}$ is necessarily unique in $C^\infty([0,\infty)\times (-1,1))$.
  \end{proof}

  \begin{lemma}[Boundedness of the energy]
    \label{lem:en0}
    Let $f,g\in C^\infty(-1,1)$ with $\|(f,g)\|_{\mc H}<\infty$. Then
    we have
    \[ \|(u_{f,g}(s,\cdot),\partial_s u_{f,g}(s,\cdot))\|_{\mc
        H}\lesssim \|(f,g)\|_{\mc H} \]
    for all $s\geq 0$.
  \end{lemma}

  \begin{proof}
    This is a simple exercise.
%     We have
%     \begin{align*}
%       \partial_y u_{f,0}(s,y)
%       &=\tfrac12 (2-e^{-s}(1+y))f'(-1+e^{-s}(1+y))e^{-s} \\
% &\quad +\tfrac12 (2-e^{-s}(1-y))f'(1-e^{-s}(1-y))e^{-s}
%     \end{align*}
%     and thus,
%     \begin{align*}
%       \int_{-1}^1
%       &(1-y^2)|\partial_y u_{f,0}(s,y)|^2 dy \\
%       &\lesssim \int_{-1}^1 (1+y)(2-e^{-s}(1+y))^2f'(-1+e^{-s}(1+y))^2
%         e^{-2s}dy \\
%       &\quad +\int_{-1}^1 (1-y)(2-e^{-s}(1-y))^2 f'(1-e^{-s}(1-y))^2
%         e^{-2s}dy \\
%       &=\int_{-1}^{-1+2e^{-s}}(1+x)(1-x)^2f'(x)^2dx+\int_{1-2e^{-s}}^1(1-x)(1+x)^2f'(x)^2
%         dx \\
%         &\lesssim \|(f,0)\|_{\mc H}^2
%     \end{align*}
%     for all $s\geq 0$.
%     Furthermore, Cauchy-Schwarz yields
%     $\|u_{0,g}(s,\cdot)\|_{L^2(-1,1)}\lesssim
%     \|g\|_{L^2(-1,1)}=\|(0,g)\|_{\mc H}$.
  \end{proof}

\subsection{Solution for odd data and Strichartz estimates}
The existence of the constant finite-energy solution $u(s,y)=1$ precludes the
possibility of Strichartz estimates. Consequently, 
we restrict ourselves to odd data $f,g\in
C^\infty(-1,1)$.
Then the solution $u_{f,g}$ is given by
  \[ u_{f,g}(s,y)=\frac12\int_{1-e^{-s}(1+y)}^{1-e^{-s}(1-y)}\left
      [(1+x)f'(x)+g(x)\right ]dx. \]
The following simple Sobolev embedding shows that the energy is strong
enough to control $L^q$, provided $q<\infty$.

\begin{lemma}
  \label{lem:LpH1}
  Let $q\in [1,\infty)$. Then we have the bound
  \[ \|f\|_{L^q(-1,1)}\lesssim \left \|(1-|\cdot|^2)^\frac12 f'\right
    \|_{L^2(-1,1)} \]
  for all odd $f\in C^1(-1,1)$ such that the right-hand side is finite.
\end{lemma}

\begin{proof}
By the fundamental theorem of calculus and the oddness of $f$, we
infer
\[ f(y)=\int_0^y f'(x)dx \]
for all $y\in (-1,1)$ 
and Cauchy-Schwarz yields
\begin{align*}
|f(y)|&\leq \int_0^{|y|} |f'(x)|dx=\int_0^{|y|}
        (1-x^2)^{-\frac12}(1-x^2)^\frac12 |f'(x)|dx \\
  &\leq \left( \int_0^{|y|} (1-x^2)^{-1}dx\right )^\frac12 \left
    \|(1-|\cdot|^2)^\frac12 f'\right\|_{L^2(-1,1)}. 
\end{align*}
Since
\[ \int_0^{|y|} (1-x^2)^{-1}dx\lesssim |\log(1-y^2)|+1 \]
and the square root of the latter function belongs to $L^q(-1,1)$ for
any $q\in [1,\infty)$,
the stated bound follows.
\end{proof}
  
  \begin{proposition}[Strichartz estimates for the free equation]
\label{prop:Strich0}
    Let $p\in [2,\infty]$ and $q\in [1,\infty)$. Then we have the
  Strichartz estimates
  \[ \|u_{f,g}\|_{L^p(0,\infty)L^q(-1,1)}\lesssim \|(f,g)\|_{\mc H} \]
  for all odd $f,g\in C^\infty(-1,1)$ with
  $\|(f,g)\|_{\mc H}<\infty$.
  \end{proposition}

  \begin{proof}
  The case $p=\infty$ is a consequence of Lemmas \ref{lem:en0} and
  \ref{lem:LpH1}. Thus, it suffices to prove the bound
  \[ \|u_{f,g}\|_{L^2(1,\infty)L^q(-1,1)}\lesssim \|(f,g)\|_{\mc H}. \]
 We first consider the case
    $g=0$.
   Then we have
   \begin{align*}
     |u_{f,0}(s,y)|&\lesssim
      \int_{1-e^{-s}(1+y)}^{1-e^{-s}(1-y)}|f'(x)|dx
                     =e^{-s}\int_{1-y}^{1+y}|f'(1-e^{-s}x)|dx \\
     &=\int_{0}^2 1_{[1-y,1+y]}(x)|e^{-s}f'(1-e^{-s}x)|dx
   \end{align*}
   for all $y\in [0,1)$ and thus, by Minkowski's inequality and the oddness of
   $u_{f,0}(s,\cdot)$,
   \begin{align*}
     \|u_{f,0}(s,\cdot)\|_{L^q(-1,1)}
     &\lesssim \left \|\int_{0}^2
       1_{[1-y,1+y]}(x)|e^{-s}
     f'(1-e^{-s}x)|dx\right \|_{L^q_y(0,1)} \\
     &\lesssim \int_{0}^2 \|1_{[1-y,1+y]}(x)\|_{L_y^q(0,1)}|e^{-s}f'(1-e^{-s}x)|dx.
   \end{align*}
Now note that $1_{[1-y,1+y]}(x)\leq 1_{[1-x,1]}(y)$ 
   for all $x\in [0,2]$ and $y\in [0,1]$.
This yields 
$\|1_{[1-y,1+y]}(x)\|_{L^q_y(0,1)}\lesssim x^{\frac{1}{q}}$
and thus,
\begin{equation*}
  \|u_{f,0}(s,\cdot)\|_{L^q(-1,1)}\lesssim
  \int_0^2 x^\frac{1}{q}|e^{-s}f'(1-e^{-s}x)|dx.
  \end{equation*}
Consequently, 
   \begin{align*}
     \|u_{f,0}\|_{L^2(1,\infty)L^q(-1,1)}
     &\lesssim \left \|\int_0^2
       x^\frac{1}{q}|e^{-s}f'(1-e^{-s}x)|dx\right\|_{L^2_s(1,\infty)}\\
       &\lesssim \int_0^2 x^{\frac{1}{q}}\|e^{-s}f'(1-e^{-s}x)\|_{L^2_s(1,\infty)}dx
   \end{align*}
   again by Minkowski's inequality.
   Now we have
   \begin{align*}
     \|e^{-s}f'(1-e^{-s}x)\|_{L^2_s(1,\infty)}^2
     &=\int_1^\infty |f'(1-e^{-s}x)|^2 e^{-2s}ds \\
     &=x^{-2}\int_{1-e^{-1}x}^1(1-\eta)|f'(\eta)|^2d\eta \\
     &\lesssim x^{-2}\|(f,0)\|_{\mc H}^2
   \end{align*}
for all $x\in (0,2]$ and in summary, we obtain
   \[ \|u_{f,0}\|_{L^2(1,\infty)L^q(-1,1)}\lesssim \|(f,0)\|_{\mc
       H}\int_{0}^2 x^{\frac{1}{q}-1}dx\lesssim \|(f,0)\|_{\mc
       H}. \]

   The case $f=0$ is much simpler and it suffices to note that
   \[ |u_{0,g}(s,y)|\leq \int_{1-e^{-s}(1+y)}^{1-e^{-s}(1-y)}
     |g(x)|dx\lesssim e^{-\frac12 s}\|g\|_{L^2(-1,1)}\lesssim
     e^{-\frac12 s}\|(0,g)\|_{\mc H} \]
   for all $s\geq 0$ and $y\in [0,1]$.
  \end{proof}

In particular, Proposition \ref{prop:Strich0} shows that the zero
solution is asymptotically stable under odd perturbations in the
energy space.

\subsection{Semigroup formulation}
For later purposes it is desirable to translate the results obtained
so far into semigroup language. First, we need to define proper
function spaces and operators.

\begin{definition}
  We set
  \[ \widetilde{\mc H}:=\{\mb f=(f_1,f_2)\in C^\infty([-1,1])\times C^\infty([-1,1]): \mb f
    \mbox{ is odd}\}. \]
  The vector space $\widetilde{\mc H}$ equipped with the inner
  product
  \[ (\mb f | \mb g)_{\mc H}:=\int_{-1}^1 (1-y^2)f_1'(y)
    \overline{g_1'(y)}dy
    +\int_{-1}^1 f_2(y)\overline{g_2(y)}dy \]
  is a pre-Hilbert space and we denote by $\mc H$ its
  completion.
  Furthermore, we consider the \emph{formal} differential expression
  \[ \mathfrak L_0\mb f(y):=
    \begin{pmatrix}
      f_2(y) \\
      (1-y^2)f_1''(y)-2yf_1'(y)-2yf_2'(y)-f_2(y)
    \end{pmatrix}
  \]
  and define the operator $\widetilde{\mb L}_0: \mc D(\widetilde{\mb
    L}_0)\subset \mc H\to\mc H$ by 
  $\mc D(\widetilde{\mb L}_0):=\widetilde{\mc H}$ and $\widetilde{\mb
    L}_0\mb f:=\mathfrak L_0 \mb f$.
\end{definition}
By construction, $\widetilde{\mb L}_0$ is a densely-defined
operator on $\mc H$.
With these definitions at hand, the initial value problem
\eqref{eq:HIV} can be written as
\[ \begin{cases}
    \partial_s\Phi(s)=\widetilde{\mb L}_0\Phi(s) & \mbox{for all }s>0 \\
    \Phi(0)=\mb f
  \end{cases}
\]
for $\Phi(s)=(u(s,\cdot),\partial_s u(s,\cdot))$ and $\mb
f=(f,g)$. The well-posedness of this initial value problem now means
that (the closure of) $\widetilde{\mb L}_0$ generates a semigroup.

\begin{lemma}
  \label{lem:S0}
  The operator $\widetilde{\mb L}_0: \mc D(\widetilde{\mb L}_0)\subset
  \mc H\to\mc H$ is closable and its closure $\mb L_0$ generates a
  strongly continuous one-parameter semigroup $\{\mb S_0(s)\in \mc
  B(\mc H): s\geq 0\}$. Furthermore, we have the estimate $\|\mb S_0(s)\mb f\|_{\mc
    H}\leq \|\mb f\|_{\mc H}$ for all $s\geq 0$ and all $\mb f\in \mc H$.
\end{lemma}

\begin{proof}
  A straightforward computation shows that
  \[ \Re \big (\widetilde{\mb L}_0\mb f\big |\mb f\big )_{\mc
      H}=-|f_2(-1)|^2-|f_2(1)|^2\leq 0 \]
  for all $\mb f=(f_1,f_2)\in \mc D(\widetilde{\mb L}_0)$ (this is
  just an instance of the energy identity Eq.~\eqref{eq:enid}). Furthermore, 
  for
  $\mb g=(g_1,g_2)\in \widetilde{\mc H}$ we set
  $g(y):=yg_1'(y)+g_1(y)+\frac12 g_2(y)$ and
  \begin{align*}
    f_1(y):=
    \frac{1}{1+y}\int_{-1}^y (1+x)g(x)dx
 +\frac{1}{1-y}\int_y^1 (1-x)g(x)dx.
  \end{align*}
  Note that $f_1$ is odd and belongs to $C^\infty([-1,1])$.
We set $\mb f:=(f_1,f_1-g_1)$.
Then we have $\mb f\in \mc D(\widetilde{\mb L}_0)$ and a
straightforward computation shows that $(1-\widetilde{\mb L}_0)\mb
f=\mb g$. Since $\mb g\in \widetilde{\mc H}$ was arbitrary, we see
that the range of $1-\widetilde{\mb L}_0$ is dense in $\mc H$ and an
application of the Lumer-Phillips theorem (see e.g.~\cite{EngNag00},
p.~83, Theorem 3.15) completes the proof.
\end{proof}

\begin{corollary}
  \label{cor:specL0}
We have $\sigma(\mb L_0)\subset \{z\in \C: \Re z\leq 0\}$.
\end{corollary}

\begin{proof}
  The statement is a
  consequence of the growth bound in Lemma \ref{lem:S0} and
  \cite{EngNag00}, p.~55, Theorem 1.10.
\end{proof}

\begin{remark}
  In fact, we have $\sigma_p(\mb L_0)=\{z\in \C: \Re z<0\}$ (and hence
  $\sigma(\mb L_0)=\{z\in \C: \Re z\leq 0\}$). This follows easily by noting
  that for any $\lambda\in \C$, the function $\mb
  f=(f_1,\lambda f_1)$ with
  \[ f_1(y):=(1+y)^{-\lambda}-(1-y)^{-\lambda} \]
  satisfies $(\lambda-\mathfrak L_0)\mb f=0$.
  However, we omit a formal proof of this result since it is not
  needed in the following.
\end{remark}

\section{The wave equation with a potential}

\noindent Now we move on to the main problem and add a potential $V\in C^\infty([-1,1])$.
In order to retain the parity symmetry, we require $V$ to be even.
That is to say, we study the initial value problem
\begin{equation}
  \label{eq:HIVV}
  \begin{cases}
    [\partial_s^2+2y\partial_s\partial_y+\partial_s-(1-y^2)\partial_y^2+2y\partial_y+V(y)]u(s,y)=0
  & (s,y)\in (0,\infty)\times (-1,1)
  \\
  u(s,y)=f(y),\qquad \partial_s u(s,y)=g(y) & (s,y)\in \{0\}\times (-1,1)
  \end{cases}.
\end{equation}
\subsection{Semigroup formulation}
We immediately switch to the semigroup picture. Note that by Lemma
\ref{lem:LpH1}, the operator $(f_1,f_2)\mapsto (0,-Vf_1)$ is bounded
on $\mc H$.

\begin{definition}
  Let $V\in C^\infty([-1,1])$ be even. Then we define the bounded
  operator $\mb L'_V: \mc H\to\mc H$ by
  \[ \mb L'_V \mb f:=
    \begin{pmatrix}
      0 \\ -Vf_1
    \end{pmatrix}.
  \]
  Furthermore, we set $\mb L_V:=\mb L_0+\mb L'_V$, where $\mb L_0$ is
  the closure of $\widetilde{\mb L}_0$, see Lemma \ref{lem:S0}.
\end{definition}
Eq.~\eqref{eq:HIVV} can be written as
\[
  \begin{cases}
    \partial_s\Phi(s)=\mb L_V\Phi(s) & \mbox{for all }s>0 \\
    \Phi(0)=(f,g)
  \end{cases}
\]
and the abstract theory immediately tells us that this initial value
problem is well-posed.

\begin{lemma}
  \label{lem:LV}
  The operator $\mb L_V: \mc D(\mb L_0)\subset\mc H\to\mc H$ generates
  a strongly continuous one-parameter semigroup $\{\mb S_V(s)\in \mc
  B(\mc H): s\geq 0\}$ and we have the bound
  \[ \|\mb S_V(s)\mb f\|_{\mc
    H}\leq e^{\|\mb L_V'\|s}\|\mb f\|_{\mc H}\] for all
  $s\geq 0$ and all $\mb f\in \mc H$.
\end{lemma}

\begin{proof}
The statement is a consequence of the bounded
perturbation theorem, see e.g.~\cite{EngNag00}, p.~158, Theorem 1.3.
\end{proof}

\subsection{Analysis of the generator}
In order to relate the semigroup formulation to the classical picture,
we need some technical results on the generator $\mb L_V$. The point is
that the latter 
is only abstractly defined as the closure of $\widetilde
{\mb L}_V:=\widetilde{\mb L}_0+\mb L_V'$.

\begin{lemma}
  \label{lem:Winf}
  Let $n\in \N$, $\delta\in (0,1)$, and
  $I_\delta:=(-1+\delta,1-\delta)$. Then we have the bound
  \[
    \|f_1\|_{W^{n,\infty}(I_\delta)}+\|f_2\|_{W^{n-1,\infty}(I_\delta)}\lesssim
    \|\mb f\|_{\mc H}+\|\widetilde{\mb L}_V^n\mb f \|_{\mc H} \]
  for all $\mb f=(f_1,f_2)\in \widetilde{\mc H}$.
\end{lemma}

\begin{proof}
   Since $f_1$ and $f_2$ are odd, we have
  $f_j(y)=\int_0^y f_j'(x)dx$, $j\in \{1,2\}$, and Cauchy-Schwarz yields
  \begin{align*}
    \|f_1\|_{L^\infty(I_\delta)}
    &\lesssim
    \|f_1'\|_{L^2(I_\delta)}\lesssim \|\mb f\|_{\mc H} \\
    \|f_2\|_{L^\infty(I_\delta)}
    &\lesssim
    \|f_2'\|_{L^2(I_\delta)}=\|[\widetilde{\mb L}_V\mb
      f]_1'\|_{L^2(I_\delta)}\lesssim
      \|\widetilde{\mb L}_V\mb f\|_{\mc
      H}.
      \end{align*}
  Furthermore,
  \[ f_1''(y)=\frac{1}{1-y^2}\left [[\widetilde{\mb L}_V\mb
      f]_2(y)+2yf_1'(y)+2yf_2'(y)+f_2(y)+V(y)f_1(y)\right] \]
  and thus, by the one-dimensional Sobolev embedding,
  \[ \|f_1'\|_{L^\infty(I_\delta)}\lesssim \|f_1'\|_{L^2(I_\delta)}+\|f_1''\|_{L^2(I_\delta)}
    \lesssim \|\mb f\|_{\mc H}+\|\widetilde{\mb L}_V\mb f\|_{\mc H}.
  \]
  This settles the case $n=1$ and from here we proceed inductively.
\end{proof}

\begin{corollary}
  \label{cor:regL0}
  Let $n\in \N$. If $\mb f\in \mc D(\mb L_V^n)$ then $\mb f$ can be
  identified with
  an odd function in 
  $C^n(-1,1)\times C^{n-1}(-1,1)$.
\end{corollary}

\begin{proof}
  Fix $\delta\in (0,1)$, set $I_\delta:=(-1+\delta,1-\delta)$,
  and let $\mb f\in \mc D(\mb L_V^n)$. Then there exists a sequence $(\mb
f_k)_{k\in \N}\subset \widetilde{\mc H}$ such that $\mb
f_k\to \mb f$ and $\widetilde{\mb L}_V^n\mb f_k\to \mb L_V^n\mb f$. In
particular, $(\mb f_k)_{k\in\N}$ and $(\widetilde{\mb L}_V^n\mb
f_k)_{k\in\N}$ are Cauchy sequences with respect to $\|\cdot\|_{\mc
  H}$. By Lemma \ref{lem:Winf} we see that $\mb f_k$ converges
to an odd
function in $C^n(I_\delta)\times C^{n-1}(I_\delta)$. 
Since $\delta\in (0,1)$ was arbitrary, $(\mb f_k)_{k\in \N}$ converges
pointwise on $(-1,1)$ to an odd function in $C^1(-1,1)\times C(-1,1)$,
which may be identified with $\mb f$.
\end{proof}

\begin{remark}
  \label{rem:identify}
From now on we will implicitly make the identification suggested in Corollary
\ref{cor:regL0}. Consequently, any $\mb f\in \mc D(\mb L_V^n)$ \emph{is}
an odd function in $C^n(-1,1)\times C^{n-1}(-1,1)$ and we have the
inclusion $\mc D(\mb
L_V^n)\subset C^n(-1,1)\times C^{n-1}(-1,1)$.
\end{remark}

\begin{corollary}
  \label{cor:actL0}
  On $\mc D(\mb L_0^2)$, $\mb L_0$ acts as a classical differential
  operator, i.e., if $\mb f\in \mc D(\mb L_0^2)\subset C^2(-1,1)\times
  C^1(-1,1)$, we have $\mb L_0\mb f=\mathfrak L_0\mb f$
  on $(-1,1)$.
\end{corollary}

\begin{proof}
  Let $\mb f\in \mc D(\mb L_0^2)$. Then there exists a sequence $(\mb
  f_k)_{k\in \N}$ with $\widetilde{\mb L}_0^n\mb f_k\to\mb
  L_0^n \mb f$ for $n\in \{0,1,2\}$. By the definition of $\widetilde{\mb L}_0$ and Lemma
  \ref{lem:Winf}, we see that $(\mathfrak L_0\mb f_k)_{k\in\N}$ converges
  pointwise on $(-1,1)$ to $\mathfrak L_0 \mb f\in C^1(-1,1)\times
  C(-1,1)$, and the latter function is identified
  with $\mb L_0\mb f$.
\end{proof}

\begin{corollary}
  Let $\mb f=(f,g)\in \widetilde{\mc H}$ and set $u(s,\cdot):=[\mb
  S_0(s)\mb f]_1$. Then we have $u=u_{f,g}$.
\end{corollary}

\begin{proof}
  We have $\mb f\in \mc D(\mb L_0^n)$ for any $n\in \N$. 
Thus, by \cite{EngNag00}, p.~124, Proposition 5.2, we obtain $\mb
S_0(s)\mb f\in \mc D(\mb L_0^n)$ for all $s\geq 0$ and any $n\in \N$.
Furthermore, since $\partial_s^n \mb S_0(s)\mb f=\mb S_0(s)\mb
L_0^n\mb f$, it follows that $u\in C^\infty([0,\infty)\times
(-1,1))$. Thus, by Corollary \ref{cor:actL0}, $u$ is a smooth
finite-energy solution of Eq.~\eqref{eq:HIV} and by Lemma
\ref{lem:dAlembert}, we must have $u=u_{f,g}$.
\end{proof}

\subsection{Spectral properties}
The special structure of the operator $\mb L_V'$
allows us to obtain important spectral information, even at this level
of generality.
First, we need a simple compactness result.

\begin{lemma}
  \label{lem:compact}
  Let $(f_n)_{n\in\N}\subset C^1(-1,1)$ be a sequence of odd functions
  that satisfy
  \[ \left \|(1-|\cdot|^2)^\frac12 f_n' \right \|_{L^2(-1,1)}\lesssim 1 \]
  for all $n\in \N$. Then there exists a subsequence of $(f_{n})_{n\in
    \N}$ that is Cauchy in $L^2(-1,1)$.
\end{lemma}

\begin{proof}
  We mimic the classical proof of the Arzel\`a-Ascoli theorem.
The set $(-1,1)\cap \Q$ is countable and dense in
  $(-1,1)$ and we write $(-1,1)\cap \Q=\{y_j: j\in \N\}$. By the fundamental theorem
  of calculus and the oddness of $f_n$, we have
  \[ f_n(y)=\int_0^y f_n'(x)dx \]
  and thus, by Cauchy-Schwarz,
  \begin{align*}
    |f_n(y)|
    &\leq\int_0^{y}|f_n'(x)|dx=\int_0^{y}(1-x^2)^{-\frac12}(1-x^2)^\frac12
      |f_n'(x)|dx \\
    &\leq \left (
      \int_0^{y}(1-x^2)^{-1}dx\right)^{1/2}
      \left \|(1-|\cdot|^2)^\frac12 f_n'\right\|_{L^2(-1,1)} \\
    &\lesssim \left |\log(1-y^2)\right |^\frac12+1
  \end{align*}
  for all $y\in (-1,1)$ and all $n\in \N$. Since $y_j\in (-1,1)$, this
  estimate shows that for each $j\in \N$, the sequence
  $(f_n(y_j))_{n\in\N}\subset \C$ is bounded. By Cantor's classical
  diagonal argument we extract a subsequence
  $(f_{n_k})_{k\in\N}$ of $(f_n)_{n\in\N}$ such that for each $j\in
  \N$, $(f_{n_k}(y_j))_{k\in\N}$ is Cauchy in $\C$.

Now note that for any $\delta\in (0,1)$, we have the bound
  \[ |f_n(x)-f_n(y)|\lesssim \delta^{-\frac12}|x-y|^\frac12 \]
  for all $x,y\in [-1+\delta,1-\delta]$ and $n\in \N$. Indeed,
  \[ f_n(x)-f_n(y)= \int_y^x f_n'(t)dt \]
  and thus, by Cauchy-Schwarz,
  \begin{align*}
    |f_n(x)-f_n(y)|
    &\lesssim |x-y|^\frac12 \left (\int_{-1+\delta}^{1-\delta}|f'_n(t)|^2
      dt\right )^{1/2}\lesssim \delta^{-\frac12}|x-y|^\frac12 \left
      \|(1-|\cdot|^2)^\frac12 f_n'\right\|_{L^2(-1,1)} \\
    &\lesssim \delta^{-\frac12}|x-y|^\frac12,
  \end{align*}
  as claimed. As a consequence of this estimate, $(f_n)_{n\in\N}$ is
  equicontinuous on $[-1+\delta,1-\delta]$ and the density of $\{y_j:
  j\in \N\}$ implies that $(f_{n_k})_{k\in\N}$ is Cauchy in
  $L^\infty(-1+\delta,1-\delta)$. 

  Now let $\epsilon\in (0,1)$. Then there exists an $N_\epsilon\in \N$ such that
  \[ \|f_{n_k}-f_{n_\ell}\|_{L^\infty(-1+\epsilon,1-\epsilon)}\leq
      \epsilon \]
    for all $k,\ell\geq N_\epsilon$. Consequently,
    \begin{align*}
      \|f_{n_k}-f_{n_\ell}\|_{L^2(-1,1)}^2
      &=\|f_{n_k}-f_{n_\ell}\|_{L^2(-1+\epsilon,1-\epsilon)}^2 \\
      &\quad
        +\|f_{n_k}-f_{n_\ell}\|_{L^2(-1,-1+\epsilon)}^2+\|f_{n_k}-f_{n_\ell}\|_{L^2(1-\epsilon,1)}^2
      \\
      &\lesssim \epsilon^\frac12
    \end{align*}
for all $k,\ell\geq N_\epsilon$ since
   \[ \|f_n\|_{L^2(-1,-1+\epsilon)}^2\lesssim
     \int_{-1}^{-1+\epsilon}|\log(1-y^2)|dy\lesssim
     \epsilon^\frac12 \]
   for all $n\in \N$ and analogously for $\|f_n\|_{L^2(1-\epsilon,1)}^2$.
\end{proof}

We continue with a simple resolvent bound. Note that this bound is just a consequence
of the fact that the operator $\mb L_V'$ maps the first component to
the second component.

\begin{lemma}
  \label{lem:LVRes}
  Let $\epsilon>0$. Then we have the bound
  \[ \|\mb L_V'\mb R_{\mb L_0}(\lambda)\mb f\|_{\mc H}\lesssim
    \frac{1}{|\lambda|}\|\mb f\|_{\mc H} \]
  for all $\lambda\in \C$ with $\Re\lambda\geq \epsilon$ and all $\mb
  f\in \mc H$.
\end{lemma}

\begin{proof}
  To begin with, let $\mb f=(f_1,f_2)\in \mc D(\mb L_0)$ and set $\mb u:=\mb
  R_{\mb L_0}(\lambda)\mb f$. Then $\mb u=(u_1,u_2)\in \mc D(\mb L_0^2)$ and
  $(\lambda-\mb L_0)\mb u=\mb f$. By Corollary \ref{cor:regL0}, $\mb u\in
  C^2(-1,1)\times C^1(-1,1)$ and Corollary \ref{cor:actL0} yields
  $(\lambda-\mathfrak L_0) \mb u=\mb f$.
  The first component of this equation reads
  $\lambda u_1-u_2=f_1$ or, equivalently,
  \[ [\mb R_{\mb L_0}(\lambda)\mb f]_1=\frac{1}{\lambda}\left (\mb
      [\mb R_{\mb L_0}(\lambda)\mb f]_2+f_1\right ). \]
  Consequently,
  \begin{align*}
    \|\mb L_V'\mb R_{\mb L_0}(\lambda)\mb f\|_{\mc H}
    &=\|V[\mb R_{\mb L_0}(\lambda)\mb f]_1\|_{L^2(-1,1)}
    \lesssim \frac{1}{|\lambda|}\left (\|[\mb R_{\mb L_0}(\lambda)\mb
      f]_2\|_{L^2(-1,1)}+\|f_1\|_{L^2(-1,1)}\right ) \\
    &\lesssim \frac{1}{|\lambda|}\left (\|\mb R_{\mb L_0}(\lambda)\mb
      f\|_{\mc H}+\|\mb f\|_{\mc H}\right )
      \lesssim \frac{1}{|\lambda|}\left (\frac{1}{\Re\lambda}\|\mb
      f\|_{\mc H}+\|\mb f\|_{\mc H}\right ) \\
    &\lesssim \frac{1}{|\lambda|}\|\mb f\|_{\mc H}
  \end{align*}
  by Lemma \ref{lem:LpH1} and \cite{EngNag00}, p.~55, Theorem
  1.10. Thus, the claim follows by density.
\end{proof}

\begin{lemma}
  \label{lem:LVcompact}
  The operator $\mb L_V': \mc H\to\mc H$ is compact. As a consequence,
  the set
  \[ \sigma(\mb L_V)\cap \{z\in\C: \Re z> 0\} \] consists of
  finitely many eigenvalues of finite algebraic multiplicity.
\end{lemma}

\begin{proof}
  Let $(\mb f_n)_{n\in \N}\subset \mc H$ be a bounded sequence and
  write $\mb f_n=(f_{n,1},f_{n,2})$. Then we have
  \[ \left \|(1-|\cdot|^2)^\frac12 f_{n,1}'\right \|_{L^2(-1,1)}\lesssim
  1 \]
for all $n\in \N$ and Lemma \ref{lem:compact} implies that
$(f_{n,1})_{n\in\N}$ has a subsequence, again denoted by $(
f_{n,1})_{n\in \N}$, that is Cauchy in $L^2(-1,1)$.
We have
\[ \|\mb L_V'\mb f_m-\mb L_V'\mb f_n\|_{\mc
    H}=\|V(f_{m,1}-f_{n,1})\|_{L^2(-1,1)}\lesssim
  \|f_{m,1}-f_{n,1}\|_{L^2(-1,1)} \]
and thus, $(\mb L_V'\mb f_n)_{n\in\N}$ has a convergent
subsequence. This shows that $\mb L_V'$ is compact.

  By Corollary \ref{cor:specL0}, $\mb R_{\mb L_0}$ is holomorphic on
  the open right half-plane $H^+:=\{z\in \C: \Re z>0\}$ and 
  the obvious identity $\lambda-\mb L_V=[\mb I-\mb L_V'\mb R_{\mb
    L_0}(\lambda)](\lambda-\mb L_0)$ shows that $\lambda\in H^+$
  belongs to $\rho(\mb L_V)$ if and only if $\mb I-\mb
  L_V'\mb R_{\mb L_0}(\lambda)$ is bounded invertible.
  By Lemmas \ref{lem:LV} and \ref{lem:LVRes} we immediately see that $\sigma(\mb
  L_V)\cap H^+$ is bounded.
  Furthermore, the map
  $\lambda\mapsto \mb L_V'\mb R_{\mb L_0}(\lambda)$ is holomorphic on
  $H^+$ and has values in the set of compact operators on the Hilbert space
  $\mc H$. Consequently, the
  analytic Fredholm theorem (see e.g.~\cite{Sim15-4}, p.~194, Theorem
  3.14.3) implies that the inverse $\lambda\mapsto [\mb I-\mb L_V'\mb
  R_{\mb L_0}(\lambda)]^{-1}$ has finitely many poles of finite order with
  finite rank residues. For every $\lambda\in \sigma(\mb L_V)\cap H^+$
  we therefore have $1\in \sigma(\mb L_V'\mb R_{\mb L_0}(\lambda))$
  and thus, $1\in \sigma_p(\mb L_V'\mb R_{\mb L_0}(\lambda))$.
  Let $\mb f_\lambda\in \mc H$ be an associated eigenfunction. Then
  $\mb R_{\mb L_0}(\lambda)\mb f_\lambda\in \mc D(\mb L_V)$ is an eigenfunction of
  $\mb L_V$ and we see that every $\lambda\in \sigma(\mb
  L_V)\cap H^+$ is an eigenvalue of $\mb L_V$ and the corresponding
  spectral projection has finite rank.
\end{proof}

Lemma \ref{lem:LVcompact} allows us to remove the unstable part of the
spectrum by a finite-rank projection.

\begin{definition}
Let $\gamma: [0,2\pi]\to\rho(\mb L_V)$ be a positively oriented,
regular, smooth,
simple closed curve that
encircles the set
$\sigma(\mb L_V)\cap \{z\in \C: \Re z> 0\}$ (the existence of such a
curve is guaranteed by Lemma \ref{lem:LVcompact}). Then we define
  \[ \mb P_V:=\frac{1}{2\pi i}\int_{\gamma} \mb R_{\mb
    L_V}(\lambda)d\lambda. \]
\end{definition}

Our goal now is to prove a set of Strichartz estimates for the
\emph{reduced semigroup} $\mb S_V(s)(\mb I-\mb P_V)$ under a suitable \emph{spectral
assumption} on $\mb L_V$. To this end, we first need to clarify the
relation between the abstract Hilbert space $\mc H$ and the standard
Lebesgue spaces.

\begin{definition}
  For $q\in [1,\infty)$, we define the Banach space $L^q_\mathrm{odd}(-1,1)$ as the
  completion of $\{f\in C^\infty([-1,1]): f \mbox{ odd}\}$ with
  respect
  to $\|\cdot\|_{L^q(-1,1)}$.
\end{definition}

\begin{lemma}
  \label{lem:HinLq}
  Let $q\in [1,\infty)$.
  Then there exists a linear, bounded, and injective map $i: \mc H\to
  L_\mathrm{odd}^q(-1,1)\times L_\mathrm{odd}^2(-1,1)$. 
\end{lemma}

\begin{proof}
  For $\mb f\in \widetilde{\mc H}$ we set $i(\mb f):=\mb f$. By Lemma
  \ref{lem:LpH1} we obtain the bound
  \[ \|i(\mb f)\|_{L^q(-1,1)\times
    L^2(-1,1)}\lesssim \|\mb f\|_{\mc H} \] for all $\mb f\in
  \widetilde{\mc H}$ and by density, $i$ extends to a linear and
  bounded map $i: \mc H\to L^q_\mathrm{odd}(-1,1)\times L_\mathrm{odd}^2(-1,1)$. To show the
  injectivity, suppose that $i(\mb f)=0$ for $\mb f\in \mc H$. Then
  there exists a sequence $(\mb f_n)_{n\in \N}\in \widetilde{\mc H}$
  such that $\mb f_n\to \mb f$ in $\mc H$ and $i(\mb f_n)\to i(\mb
  f)=0$ in $L_\mathrm{odd}^q(-1,1)\times L_\mathrm{odd}^2(-1,1)$, as
  $n\to\infty$. In particular,
  $\mb f_n \rightharpoonup \mb f$ in $\mc H$. Furthermore, since
  \[ \int_{-1}^1 (1-y^2)f'(y)g'(y)dy=-\int_{-1}^1
    (1-y^2)f(y)g''(y)dy+2\int_{-1}^1 y f(y)g'(y)dy \]
for all $f,g\in C^\infty([-1,1])$,
  we see that $\mb f_n=i(\mb f_n)\to 0$ in $L_\mathrm{odd}^q(-1,1)\times L_\mathrm{odd}^2(-1,1)$
  implies $\mb f_n\rightharpoonup 0$ in $\mc H$ and the uniqueness of
  weak limits yields $\mb f=0$.
\end{proof}

\begin{remark}
  \label{rem:HinLq}
  By Lemma \ref{lem:HinLq}, we may identify $\mb f\in \mc H$ with
  $i(\mb f)\in L_\mathrm{odd}^q(-1,1)\times L_\mathrm{odd}^2(-1,1)$ and this
  yields the continuous embedding $\mc H\hookrightarrow
  L_\mathrm{odd}^q(-1,1)\times L_\mathrm{odd}^2(-1,1)$.
\end{remark}

Now we aim for proving the following Strichartz estimates for the
reduced semigroup.

\begin{theorem}
  \label{thm:StrichS}
  Let $V\in C^\infty([-1,1])$ be even, $p\in [2,\infty]$, and $q\in
  [1,\infty)$. Furthermore,
suppose that the operator $\mb L_V$ has no eigenvalues on the imaginary
axis. Then we have the Strichartz estimates
\[ \|[\mb S_V(s)(\mb I-\mb P_V)\mb
  f]_1\|_{L^p_s(0,\infty)L^q(-1,1)}\lesssim \|(\mb I-\mb P_V)\mb f\|_{\mc H} \]
for all $\mb f\in \mc H$.
\end{theorem}

\subsection{Explicit representation of the semigroup}
First, we show that the reduced semigroup $\mb S_V(s)(\mb I-\mb P_V)$
inherits the decay from the free semigroup $\mb S_0$, up to an
$\epsilon$-loss. This follows from the celebrated Gearhart-Pr\"uss
theorem and the simple resolvent bound from Lemma \ref{lem:LVRes}.

\begin{lemma}
  \label{lem:SVeps}
  Let $\epsilon>0$. Then there exists a $C_\epsilon>0$ such that
  \[ \|\mb S_V(s)(\mb I-\mb P_V)\mb f\|_{\mc H}\leq C_\epsilon
    e^{\epsilon s}\|(\mb I-\mb P_V)\mb f\|_{\mc H} \]
  for all $s\geq 0$ and all $\mb f\in \mc H$.
\end{lemma}

\begin{proof}
We denote by $\mb L^\mathrm{st}_V$ the
part of $\mb L_V$ in $\ker\mb P_V$. Then $\mb R_{\mb
  L_V^\mathrm{st}}(\lambda)$ is the part of $\mb R_{\mb
  L_V}(\lambda)$ in $\ker \mb P_V$. By construction, $\sigma(\mb
L_V^\mathrm{st})\cap \{z\in \C: \Re z>0\}=\emptyset$ and Lemma
\ref{lem:LVRes} together with the identity $\lambda-\mb L_V=[\mb I-\mb
L_V'\mb R_{\mb L_0}(\lambda)](\lambda-\mb L_0)$ shows that
\[ \sup\{\|\mb R_{\mb L_V^\mathrm{st}}(\lambda)\|_{\mc H}:
  \Re\lambda\geq \epsilon\}<\infty. \]
Consequently, the Gearhart-Pruess Theorem, see e.g.~\cite{EngNag00},
p.~302, Theorem 1.11, implies the claim.
\end{proof}

In the following, we denote by $\mb S_V^\mathrm{st}: [0,\infty)\to \mc
B(\ker\mb P_V)$ the reduced semigroup, i.e., $\mb S_V^\mathrm{st}(s)\mb f:=\mb
S_V(s)\mb f$ for all $\mb f\in \ker\mb P_V$.
The generator of the semigroup $\mb
S_V^\mathrm{st}$ is $\mb L_V^\mathrm{st}$, the part of $\mb L_V$ in $\ker\mb P_V$.
From Lemma \ref{lem:SVeps} and
\cite{EngNag00}, p.~234, Corollary 5.15,
we obtain the representation
\[ \mb S_V^\mathrm{st}(s)\mb f=\frac{1}{2\pi
    i}\lim_{N\to\infty}\int_{\epsilon-iN}^{\epsilon+iN}
  e^{\lambda s}\mb R_{\mb L_V^\mathrm{st}}(\lambda)\mb f d\lambda \]
for any $\epsilon>0$ and $\mb f\in \mc D(\mb L_V^\mathrm{st})$.
If we set $\mb u:=(\lambda-\mb L_V^\mathrm{st})^{-1}\mb f$, we
obtain $(\lambda-\mb L_V)\mb u=\mb f$. Formally at least, this
equation is equivalent to
\begin{equation*}
  \begin{split}
    \lambda u_1(y)-u_2(y)&=f_1(y) \\
    \lambda u_2(y)-(1-y^2)u_1''(y)+2yu_1'(y)+2yu_2'(y)+u_2(y)+V(y)u_1(y)&=f_2(y)
  \end{split}
\end{equation*}
and inserting the first equation into the second one yields
\begin{equation}
  \label{eq:specODE}
  -(1-y^2)u_1''(y)+2(\lambda+1)yu_1'(y)+\lambda(\lambda+1)u_1(y)+V(y)u_1(y)=F_\lambda(y)
\end{equation}
with $F_\lambda(y):=2y f_1'(y)+(\lambda+1)
f_1(y)+f_2(y)$.
Consequently, our next goal is to solve Eq.~\eqref{eq:specODE}.

\section{The Green function}
\noindent In order to solve Eq.~\eqref{eq:specODE}, we need to first construct a
suitable fundamental system for the homogeneous equation
\begin{equation}
\label{eq:specODEhom}
 -(1-y^2)u''(y)+2(\lambda+1)yu'(y)+\lambda(\lambda+1)u(y)+V(y)u(y)=0.
\end{equation}

\subsection{Construction of a fundamental system}

In terms of $v(y):=(1-y^2)^{\frac12(\lambda+1)}u(y)$,
Eq.~\eqref{eq:specODEhom} reads
\begin{equation}
  \label{eq:specODEv}
  v''(y)+\frac{1-\lambda^2}{(1-y^2)^2}v(y)=\frac{V(y)}{1-y^2}v(y).
\end{equation}
\begin{definition}
  For $y\in (-1,1)$ and $\lambda\in \C$ we set
  \begin{align*}
    \psi_1(y,\lambda)&:=(1-y)^{\frac12(1+\lambda)}(1+y)^{\frac12(1-\lambda)}.
 \end{align*}
\end{definition}
 Note that
  \begin{equation}
    \label{eq:psij}
    \partial_y^2\psi_1(y,\lambda)+\frac{1-\lambda^2}{(1-y^2)^2}\psi_1(y,\lambda)=0
    \end{equation}
    for all $y\in (-1,1)$ and $\lambda\in \C$.
    We construct a perturbative solution to
    Eq.~\eqref{eq:specODEv} with good control of the error near the
    singularity at $y=1$.

  \begin{proposition}
    \label{prop:v1}
    There exists a solution $v(y)=v_1(y,\lambda)$ to Eq.~\eqref{eq:specODEv} of the
    form
    \[ v_1(y,\lambda)=\psi_1(y,\lambda)[1+
      a_1(y,\lambda)], \]
    where the function $a_1$ satisfies
    $|a_1(y,\lambda)|\lesssim (1-y)^\frac12 \langle\lambda\rangle^{-1}$
    for all $y\in [0,1)$ and $\lambda\in \C$ with $\Re\lambda\geq
    -\frac14$.
    Furthermore, for all $k,\ell,m\in \N_0$, there
    exists a constant $C_{k,\ell,m}>0$ such that
    \[ |\partial_\kappa^m
      \partial_\omega^\ell\partial_y^k a_1(y,\kappa+i\omega)|\leq
      C_{k,\ell,m}(1-y)^{\frac12-k}\langle\omega\rangle^{-1-\ell} \]
    for all $y\in [0,1)$, $\omega\in \R$, and $\kappa\in [-\frac14,\frac14]$.
  \end{proposition}

  \begin{proof}
    To begin with, we assume $\lambda\not= 0$ and define
    \[ \psi_0(y,\lambda):=\psi_1(y,\lambda)-\psi_1(y,-\lambda). \]
    Note that
    $W(\psi_0(\cdot,\lambda),\psi_1(\cdot,\lambda))=2\lambda$.
    Consequently,
by the variation of parameters formula and
Eq.~\eqref{eq:psij}, $v_1$ has to satisfy the integral equation
\begin{equation}
  \label{eq:Volterrav1}
  v_1(y,\lambda)=\psi_1(y,\lambda)+\int_y^1
                  \frac{\psi_0(y,\lambda)\psi_1(x,\lambda)
                  -\psi_0(x,\lambda)\psi_1(y,\lambda)}{2\lambda}
                  \frac{V(x)}{1-x^2}v_1(x,\lambda)dx
                \end{equation}
                for all $y\in [0,1)$.
Conversely, any continuous solution to Eq.~\eqref{eq:Volterrav1}
belongs to $C^2([0,1))$ and solves Eq.~\eqref{eq:specODEv}.                
We write $v_1=\psi_1h_1$. Then, Eq.~\eqref{eq:Volterrav1} is equivalent
to the Volterra equation
\begin{equation}
  \label{eq:Volterrah1}
  h_1(y,\lambda)=1+\int_y^1 K(y,x,\lambda)h_1(x,\lambda)dx
  \end{equation}
with the kernel
\begin{equation}
  \label{eq:kernelK}
  \begin{split}
  K(y,x,\lambda)&=\frac{1}{2\lambda}\left
    [\frac{\psi_0(y,\lambda)}{\psi_1(y,\lambda)}\psi_1(x,\lambda)^2
    -\psi_0(x,\lambda)\psi_1(x,\lambda)\right ]\frac{V(x)}{1-x^2} \\
  &=\frac{1}{2\lambda}\left [1-\left
      (\frac{1-y}{1+y}\right)^{-\lambda}\left
      (\frac{1-x}{1+x}\right)^\lambda\right ]V(x).
  \end{split}
  \end{equation}
We have the bound
$|K(y,x,\lambda)|\lesssim (1-x)^{-\frac14}|\lambda|^{-1}$
for all $0\leq y\leq x< 1$ and all $\lambda\in \C\setminus\{0\}$ with
$\Re\lambda\geq -\frac14$. If, in addition, $|\lambda|\geq 1$, this yields
\[ \int_0^1 \sup_{y\in (0,x)}|K(y,x,\lambda)|dx\lesssim
  |\lambda|^{-1}\lesssim 1 \]
and
the Volterra existence theorem (see e.g.~\cite{SchSofSta10}, Lemma
2.4) shows that
Eq.~\eqref{eq:Volterrah1} has a solution $h_1(\cdot,\lambda)\in L^\infty(0,1)$
satisfying
$\|h_1(\cdot,\lambda)\|_{L^\infty(0,1)}\lesssim 1$ for all $\lambda\in
\C$ with $\Re\lambda\geq -\frac14$ and $|\lambda|\geq 1$.
 It follows that $h_1(\cdot,\lambda)\in C([0,1])$ and
 \begin{align*}
   |h_1(y,\lambda)-1|
   &\lesssim \int_y^1 |K(x,y,\lambda)||h_1(x,\lambda)|dx 
     \lesssim
       |\lambda|^{-1}\|h_1(\cdot,\lambda)\|_{L^\infty(0,1)}\int_y^1
       (1-x)^{-\frac14}dx \\
   &\lesssim
   (1-y)^\frac34|\lambda|^{-1}\lesssim (1-y)^\frac34\langle\lambda\rangle^{-1},
   \end{align*}
   which implies the claimed estimate on $a_1$.

The difficulty in proving the derivative bounds in the regime
$|\lambda|\geq 1$ lies with the fact
that $\lambda=\kappa+i\omega$ appears in the exponent in
Eq.~\eqref{eq:kernelK}. Thus, it seems that differentiating
with respect to $\omega$ does not improve the decay in
$\omega$. This problem can be dealt with by a suitable change of variables.
More precisely, we consider the diffeomorphism $\varphi: (0,\infty)\to
(0,1)$, $\varphi(\xi):=\frac{1-e^{-\xi}}{1+e^{-\xi}}$, with inverse
$\varphi^{-1}(x)=-\log\frac{1-x}{1+x}$.
We write $\lambda=\kappa+i\omega$ and
it suffices to consider the case $\omega\geq 1$.
Then we may rewrite Eq.~\eqref{eq:Volterrah1} as
\begin{align*}
  h_1(\varphi(\eta),\lambda)
  &=\int_{\eta}^\infty
    K(\varphi(\eta),\varphi(\xi),\lambda)h_1(\varphi(\xi),\lambda)\varphi'(\xi)d\xi
  \\
  &=\frac{1}{\omega}\int_0^\infty
    K(\varphi(\eta),\varphi(\omega^{-1}\xi+\eta),\lambda)
\varphi'(\omega^{-1}\xi+\eta)
    h_1\big (\varphi(\omega^{-1}\xi+\eta),\lambda\big )
    d\xi \\
  &=\frac{1}{2\lambda\omega}\int_0^\infty
    \left (1-e^{-(\kappa\omega^{-1}+i)\xi}\right )V\big
    (\varphi(\omega^{-1}\xi+\eta)\big )\varphi'
    (\omega^{-1}\xi+\eta) \\
  &\qquad\qquad \times 
     h_1\big (\varphi(\omega^{-1}\xi+\eta),\lambda\big )d\xi
\end{align*}
and from this representation the derivative bounds follow inductively.

In the case $|\lambda|\leq 1$ we need to argue differently due to the
apparent singularity of $K(y,x,\lambda)$ at $\lambda=0$. In fact, this
singularity is removable because $\psi_0(y,0)=0$. In order to exploit
this, we first note that
\[
  \partial_t\psi_1(y,t\lambda)=\frac{\lambda}{2}
  \log\left(\frac{1-y}{1+y}\right)\psi_1(y,t\lambda)
  \] for $t\in \R$ and then we use the fundamental theorem of calculus to write
\[ \psi_0(y,\lambda)=\int_0^1 \partial_t
  \psi_0(y,t\lambda)dt=\lambda\widetilde\psi_0(y,\lambda) \]
with
\[ \widetilde\psi_0(y,\lambda):=\frac12\log\left(\frac{1-y}{1+y}\right )
  \int_0^1
  \left [\psi_1(y,t\lambda)+\psi_1(y,-t\lambda)\right ]dt. \]
We have the bound
\begin{align*}
  |\widetilde\psi_0(y,\lambda)|
  &\lesssim
  |\log(1-y)|\int_0^1
    \left [(1-y)^{\frac12(1+t\Re\lambda)}+(1-y)^{\frac12(1-t\Re\lambda)}\right
    ]dt \\
  &\lesssim |\log(1-y)|\left [(1-y)^{\frac12(1+\Re\lambda)}+(1-y)^{\frac12(1-\Re\lambda)}\right]
  \end{align*}
and thus,
\[ |K(y,x,\lambda)|\lesssim |\log(1-x)|(1-x)^{-\frac14}\lesssim (1-x)^{-\frac12} \]
for all $0\leq y\leq x<1$ and $\lambda\in \C$ with $|\lambda|\leq 1$, $\Re\lambda\geq
-\frac14$.
Consequently, a Volterra iteration yields the stated estimate for
$a_1$. For the derivative bounds it suffices to note that
each derivative with respect to $\omega$ or $\kappa$ produces a
singular term $\log(1-x)$ which, however, is
harmless since $|\log(1-x)|^n (1-x)^{-\frac14}\leq C_n
(1-x)^{-\frac12}$ for any $n\in \N_0$. Consequently, the derivative
bounds follow inductively.
\end{proof}

\begin{definition}
  In the following, $v_1$ always refers to the solution constructed in Proposition \ref{prop:v1}.
\end{definition}

Proposition \ref{prop:v1} shows that the error $a_1$ \emph{improves}
upon differentiation with respect to $\omega$.
On the other hand, when differentiating with respect to $y$, the
bounds get worse. Both operations have in common that taking a
derivative results in the loss of one power of the respective variable
in the estimate.
This is a crucial
property and we introduce a more economical notation to keep track of
this behavior.

    \begin{definition}
      For $\alpha,\beta\in \R$, we write $f(y,\omega)=\mc
      O((1-y)^\alpha\langle\omega\rangle^\beta)$ if for all
      $k,\ell\in \N_0$ there exist constants $C_{k,\ell}>0$ such that
      \[ |\partial_\omega^\ell\partial_y^k f(y,\omega)|\leq
        C_{k,\ell}(1-y)^{\alpha-k}\langle\omega\rangle^{\beta-\ell}
      \]
      in a range of the variables $y$ and $\omega$ that is specified
      explicitly or follows from the context. In other words, the $\mc
      O$-terms may be formally differentiated. Such functions are
      said to be of \emph{symbol type}. We also use self-explanatory
      variants of this notation.
    \end{definition}

\begin{remark}
  \label{rem:Wv1}
  By Proposition \ref{prop:v1} we have, with $\omega=\Im\lambda$,
  \begin{align*}
    W(v_1(\cdot,\lambda),v_1(\cdot,-\lambda))
    &=W(\psi_1(\cdot,\lambda),\psi_1(\cdot,-\lambda))[1+\mc
      O((1-y)^\frac12\langle\omega\rangle^{-1})] \\
    &\quad +\psi_1(y,\lambda)\psi_1(y,-\lambda)\mc
      O((1-y)^{-\frac12}\langle\omega\rangle^{-1}) \\
    &=2\lambda[1+\mc O((1-y)^\frac12\langle\omega\rangle^{-1})]+\mc
      O((1-y)^\frac12\langle\omega\rangle^{-1})
  \end{align*}
  for all $y\in [0,1)$ and $\lambda\in \C$ with $|\Re\lambda|\leq \frac14$.
  This expression is in fact independent of $y$ and thus, we may
  evaluate it at $y=1$ which yields
  \[ W(v_1(\cdot,\lambda),v_1(\cdot,-\lambda))=2\lambda. \]
\end{remark}
    
The bounds on the derivatives of $v_1$ are sufficient for our purposes
and easy to work with
but
certainly not optimal, as the following result shows.

\begin{lemma}
  \label{lem:v1smooth}
  Let $\Re\lambda\geq -\frac14$. Then we have
  \[ \frac{v_1(\cdot,\lambda)}{\psi_1(\cdot,\lambda)}\in
    C^\infty([0,1]). \]
\end{lemma}

\begin{proof}
As in the proof of Proposition \ref{prop:v1} we write $v_1=\psi_1 h_1$ and
from Eqs.~\eqref{eq:Volterrah1} and \eqref{eq:kernelK}, we obtain
\[ \partial_y h_1(y,\lambda)=-\frac{1}{1-y^2}\left
    (\frac{1-y}{1+y}\right)^{-\lambda}
\int_y^1 \left (\frac{1-x}{1+x}\right )^\lambda
V(x)h_1(x,\lambda)dx. \]
The change of variables $x=y+t(1-y)$ yields
\[ \partial_y h_1(y,\lambda)=-(1+y)^{\lambda-1}\int_0^1 \left
    (\frac{1-t}{1+y+t(1-y)}\right)^\lambda
  V\big (y+t(1-y)\big )h_1\big (y+t(1-y),\lambda\big )dt \]
and from this expression the statement follows inductively.
\end{proof}

 The solution $v_1$ is sufficient to
  construct the Green function for Eq.~\eqref{eq:specODE}.

  \begin{definition}
   For $y\in [0,1)$ and $\lambda\in\C$ with $\Re\lambda\geq -\frac14$
   we set
   \[
     u_1(y,\lambda):=(1-y^2)^{-\frac12(1+\lambda)}v_1(y,\lambda). \]
   Furthermore, for $|\Re\lambda|\leq \frac14$, we define
    \[  u_0(y,\lambda):=(1-y^2)^{-\frac12(1+\lambda)}\left [v_1(0,-\lambda)
                v_1(y,\lambda)-v_1(0,\lambda) v_1(y,-\lambda) \right ]. \]
  \end{definition}

\subsection{Regularity theory}
  
We take up the opportunity to establish the link between
$\Sigma_V$, see Definition \ref{def:SigmaV}, and the spectrum of $\mb
L_V$.
The key observation in this respect is a regularity result for
the operator $\mb L_V$.

\begin{lemma}
  \label{lem:regef}
For any $\lambda\in \C$ with $\Re\lambda\geq 0$, we have
$\ker(\lambda-\mb L_V)\subset \widetilde{\mc H}$.
\end{lemma}

\begin{proof}
  Let $\mb f\in \ker(\lambda-\mb L_V)$, i.e., $\mb f\in \mc D(\mb
  L_0)$ and $\mb L_V \mb f=\lambda \mb f$. 
Inductively, this implies that $\mb f\in \mc D(\mb L_V^n)$ 
for any $n\in \N_0$. By Corollary \ref{cor:regL0} and Remark
  \ref{rem:identify}, $\mb f\in C^\infty(-1,1)\times C^\infty(-1,1)$
  and Corollary \ref{cor:actL0} shows that $(\lambda-\mathfrak L_0-\mb
  L_V')\mb f=0$. As a consequence, $f_2=\lambda f_1$ and 
$f_1$ is an odd solution of
  Eq.~\eqref{eq:specODEhom} on $(-1,1)$. Hence, it remains to show
  that $f_1\in C^\infty([-1,1])$.
  By definition,
  \[
    u_1(y,\lambda)=(1-y^2)^{-\frac12(1+\lambda)}
    \psi_1(y,\lambda)\frac{v_1(y,\lambda)}{\psi_1(y,\lambda)}
    =(1+y)^{-\lambda}\frac{v_1(y,\lambda)}{\psi_1(y,\lambda)}
  \]
  and Lemma \ref{lem:v1smooth} shows that $u_1(\cdot,\lambda)\in C^\infty([0,1])$.
Furthermore, by Proposition \ref{prop:v1},
  $u_1(y,\lambda)=(1+y)^{-\lambda}[1+
  O((1-y)^\frac12\langle\lambda\rangle^{-1})]$. In particular, there
  exists a $c\in (0,1)$ such that $|u_1(y,\lambda)|>0$ for
  all $y\in [c,1]$
and we set
\[ \widetilde u_1(y,\lambda):=u_1(y,\lambda)\int_c^y
  \frac{(1-x^2)^{-1-\lambda}}{u_1(x,\lambda)^2}dx. \]
Then 
$\{u_1(\cdot,\lambda),\widetilde
u_1(\cdot,\lambda)\}$ is a fundamental system for
Eq.~\eqref{eq:specODEhom} on $[c,1)$. As a consequence, there exist
constants $a,b\in \C$ such that
\[ f_1(y)=a u_1(y,\lambda)+b\widetilde u_1(y,\lambda) \]
for $y\in [c,1)$.
We have
$|\partial_y \widetilde u_1(y,\lambda)|\gtrsim (1-y)^{-1-\Re\lambda}$
for $y\in [c,1)$
and thus,
\[ \int_c^1 (1-y^2)|\partial_y \widetilde u_1(y,\lambda)|^2dy
\gtrsim \int_c^1 (1-y)^{-1-2\Re\lambda}dy=\infty. \]
Consequently, since $\|\mb f\|_{\mc H}<\infty$, we must have $b=0$ and
therefore, $f_1\in C^\infty([-1,1])$.
\end{proof}

\begin{lemma}
  \label{lem:SigmaV}
  We have
  \[ \Sigma_V=\sigma_p(\mb L_V)\cap \{z\in \C: \Re z\geq 0\}. \]
\end{lemma}

\begin{proof}
  Let $\lambda\in \Sigma_V$. Then $\Re\lambda\geq 0$ and there exists
  a nontrivial, odd $f_\lambda\in C^\infty([-1,1])$ that satisfies
  Eq.~\eqref{eq:specODEhom} for all $y\in (-1,1)$. We set $\mb
  f:=(f_\lambda,\lambda f_\lambda)$. Then $\mb f\in \widetilde{\mc H}$
 and $(\lambda-\widetilde{\mb L}_0-\mb L_V')\mb f=0$,
  which implies that $\lambda\in \sigma_p(\mb L_V)$.

  Conversely, if $\Re\lambda\geq 0$ and $\lambda\in \sigma_p(\mb
  L_V)$, Lemma \ref{lem:regef} implies that there exists a nontrivial
  $\mb f=(f_1, f_2)\in \widetilde{\mc H}$ such that
  $(\lambda-\widetilde{\mb L}_0-\mb L_V')\mb f=0$. In other words,
  $f_2=\lambda f_1$ and $f_1$ is a nontrivial solution of
  Eq.~\eqref{eq:specODEhom}. Consequently, $\lambda\in \Sigma_V$.
\end{proof}

Next, we relate the point spectrum of $\mb L_V$ to the value of
$u_1(y,\lambda)$ at $y=0$.
  
\begin{lemma}
  \label{lem:specpLV}
  Let $\lambda\in \C$ with $\Re\lambda \geq 0$. If
  $u_1(0,\lambda)=0$ then $\lambda\in
  \sigma_p(\mb L_V)$. 
\end{lemma}

\begin{proof}
  The function $u_1(\cdot,\lambda)$ satisfies Eq.~\eqref{eq:specODEhom}
  for all $y \in [0,1)$ and by evaluation at $y=0$, we find
  inductively that
  $\partial_y^{2k}u_1(y,\lambda)|_{y=0}=0$ for all $k\in \N_0$ (here
  the assumption $u_1(0,\lambda)=0$ enters).
  We extend $u_1(\cdot,\lambda)$ to $[-1,1]$ by setting $u_1(-y,\lambda):=-u_1(y,\lambda)$ for $y\in
  [0,1]$. Then $u_1(\cdot,\lambda)\in C^\infty([-1,1])$,
  $u_1(\cdot,\lambda)$ is odd
  and satisfies Eq.~\eqref{eq:specODEhom} for all
  $y\in (-1,1)$. This means that $\lambda\in \Sigma_V$ and Lemma
  \ref{lem:SigmaV} finishes the proof.
\end{proof}

  \subsection{Construction of the Green function}

    In order to construct the Green function, we need a more explicit expression for the Wronskian of $u_0$
    and $u_1$.
    Note carefully that
    this is the place where the spectral assumption enters.

    \begin{lemma}
      \label{lem:W}
      We have
      \[  W(u_0(\cdot,\lambda),u_1(\cdot,\lambda))(y)=
        2\lambda u_1(0,\lambda)(1-y^2)^{-1-\lambda} \]
      for all $\lambda\in \C$ with $|\Re\lambda|\leq \frac14$.
      Furthermore, if $\mb L_V$ has no eigenvalues on the imaginary
      axis, there exists an $\epsilon_0>0$ such that
      $|u_1(0,\lambda)|\gtrsim 1$ for all
      $\lambda\in \C$ with $\Re\lambda\in [0,\epsilon_0]$. 
    \end{lemma}

    \begin{proof}
      By definition and Remark \ref{rem:Wv1}, we have
     \begin{align*}
        W(u_0(\cdot,\lambda),u_1(\cdot,\lambda))(y)
        &=(1-y^2)^{-1-\lambda}W(
          v_1(0,-\lambda)v_1(\cdot,\lambda)-v_1(0,\lambda)
          v_1(\cdot,-\lambda), v_1(\cdot,\lambda)) \\
        &=-v_1(0,\lambda)W(v_1(\cdot,-\lambda),v_1(\cdot,\lambda))(1-y^2)^{-1-\lambda}
       \\
        &=2\lambda v_1(0,\lambda)(1-y^2)^{-1-\lambda} \\
       &=2\lambda u_1(0,\lambda)(1-y^2)^{-1-\lambda}
     \end{align*}
        for all $\lambda\in \C$ with
        $|\Re\lambda|\leq \frac14$.

        By assumption and Lemma \ref{lem:LVcompact}, there exists an
        $\epsilon_0>0$ such that there are no eigenvalues of $\mb L_V$
        in the strip $\{z\in \C: \Re z\in
        [0,\epsilon_0]\}$. Consequently, by Lemma \ref{lem:specpLV},
        $u_1(0,\lambda)\not= 0$ for all $\lambda\in \C$ with
        $\Re\lambda\in [0,\epsilon_0]$ and, since
        $u_1(0,\lambda)=1+O(\langle\lambda\rangle^{-1})$ by Proposition
        \ref{prop:v1}, the claim follows.
      \end{proof}

      \begin{definition}
       For any $\lambda\in \C$ with $\Re\lambda\in (0,\frac14]$ and
       $\lambda\not\in\sigma_p(\mb L_V)$, we set
    \[ G_V(y,x,\lambda):=\frac{-1}{(1-x^2)W(u_0(\cdot,\lambda), u_1(\cdot,\lambda))(x)}
      \begin{cases}
        u_0(y,\lambda)u_1(x,\lambda) & 0\leq y\leq x<1 \\
        u_1(y,\lambda)u_0(x,\lambda) & 0\leq x\leq y<1
      \end{cases}.
    \]
\end{definition}

  \begin{lemma}
    \label{lem:Green}
    There exists an $\epsilon_0>0$ such that any $\lambda\in \C$ with
    $\Re\lambda\in (0,\epsilon_0]$ belongs to $\rho(\mb L_V)$ and for any
   $\mb f=(f_1,f_2)\in \widetilde{\mc H}$ 
 we have 
 \[ \mb R_{\mb L_V}(\lambda)\mb f(y)=\begin{pmatrix}
     \int_0^1
      G_V(y,x,\lambda)
      [2xf_1'(x)+(\lambda+1)f_1(x)+f_2(x)]dx \\
    \lambda \int_0^1
      G_V(y,x,\lambda)
      [2xf_1'(x)+(\lambda+1)f_1(x)+f_2(x)]dx-f_1(y)
\end{pmatrix}
    \]
    for $y\in [0,1)$.
  \end{lemma}

  \begin{proof}
    By Lemma \ref{lem:LVcompact}, it follows that there exists an
    $\epsilon_0>0$ such that $\Re\lambda\in (0,\epsilon_0]$ implies
    $\lambda\in \rho(\mb L_V)$.
    Recall from the proof of Lemma \ref{lem:regef} that
    $u_1(\cdot,\lambda)\in C^\infty([0,1])$. Furthermore, 
    \begin{align*}
      (1-y^2)^{-\frac12(1+\lambda)}v_1(y,-\lambda)
      =(1-y^2)^{-\frac12(1+\lambda)}\psi_1(y,-\lambda)
      \frac{v_1(y,-\lambda)}{\psi_1(y,-\lambda)}=(1-y)^{-\lambda}\frac{v_1(y,-\lambda)}{\psi_1(y,-\lambda)}
    \end{align*}
    and by Lemma \ref{lem:v1smooth} we see that
    \[ u_0(y,\lambda)=
      v_1(0,-\lambda)u_1(y,\lambda)+(1-y)^{-\lambda}h(y,\lambda), \]
    where $h(\cdot,\lambda)\in C^\infty([0,1])$.
    For brevity we set $F_\lambda(x):=2xf_1'(x)+(\lambda+1)f_1(x)+f_2(x)$.
    Then, by Lemma \ref{lem:W}, we have
    \[ \int_0^1 G_V(y,x,\lambda)F_\lambda (x)dx=-\frac{1}{2\lambda u_1(0,\lambda)}\sum_{k=1}^4 I_k(y) \]
    with
    \begin{align*}
      I_1(y)
      &:=v_1(0,-\lambda)u_1(y,\lambda)\int_y^1
        (1-x^2)^\lambda u_1(x,\lambda)F_\lambda(x)dx \\
      I_2(y)
      &:=h(y,\lambda)
       (1-y)^{-\lambda}\int_y^1 (1-x^2)^\lambda
        u_1(x,\lambda)F_\lambda(x)dx \\
      I_3(y)
      &:=v_1(0,-\lambda)u_1(y,\lambda)\int_0^y
        (1-x^2)^\lambda u_1(x,\lambda)F_\lambda(x)dx \\
      I_4(y)
      &:=u_1(y,\lambda)\int_0^y (1+x)^\lambda
        h(x,\lambda)F_\lambda(x)dx. 
    \end{align*}
   By assumption, $F_\lambda\in C^\infty([0,1])$ and therefore, $I_4\in
   C^\infty([0,1])$.
   Furthermore,
   \[ I_1(y)+I_3(y)=v_1(0,-\lambda)
     u_1(y,\lambda)\int_0^1
    (1-x^2)^\lambda u_1(x,\lambda)F_\lambda(x)dx \]
  and thus, $I_1+I_3\in C^\infty([0,1])$.
  Finally, the change of variables $x=y+t(1-y)$ yields
\[ I_2(y)=h(y,\lambda)(1-y)\int_0^1 (1-t)^\lambda[1+y+t(1-y)]^\lambda
  u_1(y+t(1-y),\lambda)
  F_\lambda(y+t(1-y))dt \]
and thus, $I_2\in C^\infty([0,1])$.
In summary, the function $w_\lambda(y):=\int_0^1 G_V(y,x,\lambda)F_\lambda(x)dx$
belongs to $C^\infty([0,1])$ and by construction, $w_\lambda$
satisfies Eq.~\eqref{eq:specODE} for $y\in [0,1]$.
Furthermore, $w_\lambda(0)=0$ and by the oddness of $F_\lambda$, we
find inductively from Eq.~\eqref{eq:specODE} that
$w_\lambda^{(2k)}(0)=0$ for all $k\in \N_0$. This means that
$w_\lambda$ extends to an odd function in $C^\infty([-1,1])$ and $w_\lambda$
satisfies Eq.~\eqref{eq:specODE} for all $y\in [-1,1]$.
As a consequence, $\mb u:=(w_\lambda,\lambda w_\lambda-f_1)$
belongs to $\widetilde{\mc H}$ and satisfies $(\lambda-\mb L_V)\mb
u=\mb f$, which means that $\mb u=\mb R_{\mb L_V}(\lambda)\mb f$.
\end{proof}

\subsection{Time evolution on the unstable subspace}

By now we have collected enough information so that we can prove the
first part of Theorem \ref{thm:main}, which is a consequence of the
following result combined with Lemmas \ref{lem:SVeps} and \ref{lem:SigmaV}.

\begin{lemma}
  \label{lem:evolunstable}
  We have $\rg \mb P_V\subset C^\infty(-1,1)\times C^\infty(-1,1)$ and
  for every $\lambda\in \sigma(\mb L_V)\cap \{z\in \C: \Re
  z>0\}=:\sigma_u(\mb L_V)$
  there exists a number $n(\lambda)\in \N_0$
such that
\[ \mb S_V(s)\mb f=\sum_{\lambda\in \sigma_u(\mb L_V)}e^{\lambda s}
  \sum_{k=0}^{n(\lambda)}\frac{s^k}{k!} (\mb L_V-\lambda)^k \mb f
\]
for all $\mb f\in \rg \mb P_V$ and all $s\geq 0$.
\end{lemma}

\begin{proof}
  By Lemma \ref{lem:LVcompact}, $\sigma_u(\mb L_V)$ is finite and
  consists of eigenvalues with finite algebraic multiplicities. For
  each $\lambda\in \sigma_u(\mb L_V)$, let $\mb P_{V,\lambda}$ be the
  corresponding spectral projection. Then
  \[ \rg \mb P_V=\bigoplus_{\lambda\in \sigma_u(\mb L_V)}\rg \mb
    P_{V,\lambda}. \]
  Denote by $\mb L_{V,\lambda}$ the part of $\mb L_V$ in the
  finite-dimensional subspace $\rg\mb P_{V,\lambda}$.
Clearly, $\rg \mb P_{V,\lambda}\subset \mc D(\mb L_V)$ and for any
$\mb f\in \rg\mb P_{V,\lambda}$, we have $\mb L_{V,\lambda}\mb f=\mb
L_V\mb f\in
\rg \mb P_{V,\lambda}\subset \mc D(\mb L_V)$. Inductively, this implies $\rg \mb
P_{V,\lambda}\subset \mc D(\mb L_V^n)$ for any $n\in \N$ and Corollary
\ref{cor:regL0} shows that $\rg \mb P_{V,\lambda}\subset
C^\infty(-1,1)\times C^\infty(-1,1)$.

Let $\mb S_{V,\lambda}(s)$ be the part of $\mb S_V(s)$ in $\rg\mb
P_{V,\lambda}$ and set $\widetilde{\mb S}_{V,\lambda}(s):=e^{-\lambda
  s}\mb S_{V,\lambda}(s)$. Then $\widetilde{\mb S}_{V,\lambda}(s)$ is a
semigroup on $\rg\mb P_{V,\lambda}$ with generator $\mb
L_{V,\lambda}-\lambda$.
Since $\sigma(\mb L_{V,\lambda}-\lambda)=\{0\}$ and $\dim \rg\mb
  P_{V,\lambda}<\infty$, it follows that $\mb L_{V,\lambda}-\lambda$
  is nilpotent and there exists an $n(\lambda)\in \N$ such that
  $(\mb L_{V,\lambda}-\lambda)^{n(\lambda)}=\mb 0$.
  Note that $\partial_s^n \widetilde{\mb S}_{V,\lambda}(s)\mb
  f=\widetilde{\mb S}_{V,\lambda}(s)(\mb L_{V,\lambda}-\lambda)^n \mb
  f$ for all $n\in \N_0$ and $\mb f\in \rg\mb P_{V,\lambda}$. Consequently, 
  \[ \partial_s^{n(\lambda)}\widetilde{\mb S}_{V,\lambda}(s)\mb f=
    \widetilde{\mb S}_{V,\lambda}(s)(\mb
    L_{V,\lambda}-\lambda)^{n(\lambda)}\mb f=0 \]
  and integrating this equation yields
\[ \widetilde{\mb S}_{V,\lambda}(s)\mb f=\sum_{k=0}^{n(\lambda)}\frac{s^k}{k!}(\mb
  L_{V,\lambda}-\lambda)^k\mb f=
  \sum_{k=0}^{n(\lambda)}\frac{s^k}{k!}(\mb
  L_{V}-\lambda)^k\mb f. \]
Summation over all $\lambda\in \sigma_u(\mb L_V)$ finishes the proof.
\end{proof}

\section{Strichartz estimates}
  
\noindent In order to separate the free evolution from the effect of the potential,
we introduce suitable operators that account for the difference.
\begin{definition}
  For any $\epsilon>0$ and $f\in C([0,1])$, we set
  \begin{align*} [T_{\epsilon}(s)f](y)&:=\frac{1}{2\pi
      i}\lim_{N\to\infty}\int_{\epsilon-iN}^{\epsilon+iN}e^{\lambda s}
   \int_0^1 \left [G_V(y,x,\lambda)-G_0(y,x,\lambda)\right] f(x)dx
    d\lambda \\
    [\dot T_{\epsilon}(s)f](y)&:=\frac{1}{2\pi
      i}\lim_{N\to\infty}\int_{\epsilon-iN}^{\epsilon+iN}\lambda e^{\lambda s}
   \int_0^1 \left [G_V(y,x,\lambda)-G_0(y,x,\lambda)\right] f(x)dx
    d\lambda
    \end{align*}
\end{definition}

The key result for the Strichartz estimates are the following bounds
on $T_\epsilon$ and $\dot T_\epsilon$.
\begin{theorem}
  \label{thm:T}
  Let $p\in [2,\infty]$ and $q\in [1,\infty)$. Then there exists an
  $\epsilon_0>0$ such that
  \begin{align*} \|e^{-\epsilon s}T_\epsilon(s) f\|_{L^p_s(0,\infty)L^q(0,1)}&\lesssim
    \left \|(1-|\cdot|^2)^\frac12 f \right \|_{L^2(0,1)} \\
    \|e^{-\epsilon s}\dot T_\epsilon(s) f\|_{L^p_s(0,\infty)L^q(0,1)}&\lesssim
    \left \|(1-|\cdot|^2)^\frac12 f' \right \|_{L^2(0,1)}+\|f\|_{L^2(0,1)}
    \end{align*}
  for all $\epsilon\in (0,\epsilon_0]$ and $f\in C^1([0,1])$.
\end{theorem}

We now reduce the proof of Theorem \ref{thm:StrichS} to Theorem
\ref{thm:T}. The rest of this section is then devoted to the proof of
Theorem \ref{thm:T}.

\begin{lemma}
  Assume that Theorem \ref{thm:T} holds. Then Theorem
  \ref{thm:StrichS} follows.
\end{lemma}

\begin{proof}
Let $\mb f\in \mc D(\mb L_V)$. Then $(\mb
I-\mb P_V)\mb f\in \mc D(\mb L_V)$ and by \cite{EngNag00}, p.~234,
Corollary 5.15, we obtain
\begin{align*}
  \mb S_V(s)(\mb I-\mb P_V)\mb f
  &=\mb S_V^\mathrm{st}(s)(\mb I-\mb P_V)\mb f
    =\frac{1}{2\pi i}\lim_{N\to\infty}\int_{\epsilon-iN}^{\epsilon+iN}
    e^{\lambda s}\mb R_{\mb L_V^\mathrm{st}}(\lambda)(\mb I-\mb
    P_V)\mb fd\lambda
\end{align*}
for all $\epsilon>0$.
By Lemma \ref{lem:Green} there exists an $\epsilon_0>0$ such that for
all $\epsilon\in (0,\epsilon_0]$,
\begin{align*}   \mb S_V(s)(\mb I-\mb P_V)\mb f
  =\frac{1}{2\pi i}\lim_{N\to\infty}\int_{\epsilon-iN}^{\epsilon+iN}
    e^{\lambda s}\mb R_{\mb L_V}(\lambda)(\mb I-\mb
    P_V)\mb fd\lambda.
\end{align*}
Now we set $\mc H_0:=(\ker \mb P_V \cap \mc D(\mb L_V))+\widetilde{\mc
  H}$, $\mc Y:=L^p_\mathrm{loc}((0,\infty),L^q(0,1))$, and for
$\epsilon\in (0,\epsilon_0]$ and $\mb f\in \mc H_0$ we define 
\[ \Phi_\epsilon(\mb f)(s):=\frac{1}{2\pi
      i}\lim_{N\to\infty}\int_{\epsilon-iN}^{\epsilon+iN}
    e^{\lambda s}[\mb R_{\mb L_V}(\lambda)\mb f]_1
                                    d\lambda
                                    -[\mb S_0(s)\mb f]_1.
                                  \]
We claim that $\Phi_\epsilon$ maps $\mc H_0$ into $\mc Y$.                                  
Indeed, for $\mb f\in \ker\mb P_V \cap \mc D(\mb L_V)$, we have
$\Phi_\epsilon(\mb f)(s)=[\mb S_V(s)\mb f]_1-[\mb S_0(s)\mb f]_1$
and thus, $\Phi_\epsilon(\mb f) \in \mc Y$ by Remark \ref{rem:HinLq}. 
Furthermore, for $\mb f=(f_1,f_2)\in \widetilde{\mc H}$, Lemma
\ref{lem:Green} shows that
\[    \Phi_\epsilon(\mb f)(s)=T_\epsilon(s) \big
  (2|\cdot|f_1'+f_1+f_2 \big )
  +\dot T_\epsilon(s)
  f_1  \]
and Theorem \ref{thm:T} yields the bound
\begin{equation}
  \label{eq:bdPhieps}
  \|e^{-\epsilon s}\Phi_\epsilon(\mb f)(s)
  \|_{L^p_s(0,\infty)L^q(0,1)}\lesssim \|\mb f\|_{\mc H}.
  \end{equation}
Consequently, $\Phi_\epsilon(\mb f)\in \mc Y$ for all $\mb f\in \mc
H_0$, as claimed.
By density, $\Phi_\epsilon$ uniquely extends to a map $\Phi_\epsilon:
\mc H\to \mc Y$ and the bound \eqref{eq:bdPhieps} holds for all $\mb f\in \mc
H$. For $\mb f\in \ker\mb P_V\cap \mc D(\mb L_V)$ we obtain
\begin{align*}
  \|e^{-\epsilon s}[\mb S_V(s)\mb f]_1\|_{L^p_s(0,\infty)L^q(0,1)}
&\leq \|e^{-\epsilon s}\Phi_\epsilon(\mb
f)(s)\|_{L^p_s(0,\infty)L^q(0,1)}
                                                                     +\|e^{-\epsilon s}[\mb S_0(s)\mb f]_1\|_{L^p_s(0,\infty)L^q(0,1)} \\
  &\lesssim \|\mb f\|_{\mc H}
\end{align*}
by Proposition \ref{prop:Strich0} and monotone convergence yields
\[ \|[\mb S_V(s)\mb f]_1\|_{L^p_s(0,\infty)L^q(0,1)}\lesssim \|\mb
  f\|_{\mc H} \]
  which, by density, extends to all $\mb f\in \ker \mb P_V$. 
\end{proof}

\subsection{Analysis of the operator $T_\epsilon$}
First, we identify the integral kernel of $T_\epsilon$.

  \begin{lemma}
    \label{lem:GV-G0}
Let $\lambda\in \C\setminus \{0\}$ with $|\Re\lambda|\leq \frac14$ and
$\omega=\Im\lambda$. Then we have
   \[ G_V(y,x,\lambda)-G_0(y,x,\lambda)=\sum_{j=1}^4
     G_{V,j}(y,x,\lambda)=\sum_{j=1}^4 \widetilde G_{V,j}(y,x,\lambda),
   \]
   where
   \begin{align*}
     G_{V,1}(y,x,\lambda)&=1_{[0,1]}(x-y)\lambda^{-1}(1+y)^{-\lambda}(1-x)^{\lambda}\mc
                       O((1-y)^0(1-x)^0\langle
                       \omega\rangle^{-1}) \\
     G_{V,2}(y,x,\lambda)&=1_{[0,1]}(x-y)\lambda^{-1}(1-y)^{-\lambda}(1-x)^{\lambda}\mc
                       O((1-y)^0(1-x)^0\langle
                       \omega\rangle^{-1}) \\
              G_{V,3}(y,x,\lambda)&=1_{[0,1]}(y-x)\lambda^{-1}(1+y)^{-\lambda}(1-x)^{\lambda}\mc
                       O((1-y)^0(1-x)^0\langle
                       \omega\rangle^{-1}) \\
     G_{V,4}(y,x,\lambda)&=1_{[0,1]}(y-x)\lambda^{-1}(1+y)^{-\lambda}(1+x)^{\lambda}\mc
                       O((1-y)^0(1-x)^0\langle
                       \omega\rangle^{-1}) 
   \end{align*}
   as well as
   \begin{align*}
    \widetilde G_{V,1}(y,x,\lambda)&=1_{[0,1]}(x-y)(1-x)^\lambda\int_0^1 (1+y)^{-t\lambda}\mc
      O((1-y)^0(1-x)^0\langle\omega\rangle^0 t^0)dt \\
    \widetilde G_{V,2}(y,x,\lambda)
                                   &=1_{[0,1]}(x-y)\langle\log(1-y)\rangle(1-x)^\lambda \\
     &\quad\times\int_0^1 (1-y)^{-t\lambda}\mc
      O((1-y)^0(1-x)^0\langle\omega\rangle^0 t^0)dt \\
     \widetilde G_{V,3}(y,x,\lambda)
     &=1_{[0,1]}(y-x)(1+y)^{-\lambda}(1-x)^\lambda \\
       &\quad\times \int_0^1 (1+x)^{(1-t)\lambda} 
         \mc O((1-y)^0(1-x)^0\langle\omega\rangle^0 t^0)dt \\
     \widetilde G_{V,4}(y,x,\lambda)
     &=1_{[0,1]}(y-x)\langle\log(1-x)\rangle(1+y)^{-\lambda}(1+x)^\lambda \\
       &\quad\times \int_0^1 (1-x)^{(1-t)\lambda} 
         \mc O((1-y)^0(1-x)^0\langle\omega\rangle^0 t^0)dt 
   \end{align*}
   for all $y,x\in [0,1)$.
  \end{lemma}

  \begin{proof}
By definition and Proposition \ref{prop:v1},
 \begin{align*}
   u_1(y,\lambda)
   &=(1-y^2)^{-\frac12(1+\lambda)}v_1(y,\lambda)=(1-y^2)^{-\frac12(1+\lambda)}
   \psi_1(y,\lambda)[1+\mc
     O((1-y)^\frac12\langle\omega\rangle^{-1})] \\
   &=(1+y)^{-\lambda}[1+\mc O((1-y)^\frac12\langle\omega\rangle^{-1})]
 \end{align*}
 as well as
 \begin{align*}
   u_0(y,\lambda)
   &=(1-y^2)^{-\frac12(1+\lambda)}\left [v_1(0,-\lambda)
                v_1(y,\lambda)-v_1(0,\lambda) v_1(y,-\lambda) \right ]
   \\
   &=(1-y^2)^{-\frac12(1+\lambda)}\psi_1(y,\lambda)[1+\mc
     O((1-y)^0\langle\omega\rangle^{-1})] \\
     &\quad -(1-y^2)^{-\frac12(1+\lambda)}\psi_1(y,-\lambda)[1+\mc
       O((1-y)^0\langle\omega\rangle^{-1})] \\
   &=(1+y)^{-\lambda}[1+\mc
     O((1-y)^0\langle\omega\rangle^{-1})]
     -(1-y)^{-\lambda}[1+\mc
       O((1-y)^0\langle\omega\rangle^{-1})].
 \end{align*}
 Finally, by Lemma \ref{lem:W},
 \begin{align*}
   \frac{1}{(1-x^2)W(u_0(\cdot,\lambda),u_1(\cdot,\lambda))(x)}
   &=\frac{(1-x^2)^\lambda}{2\lambda u_1(0,\lambda)}
     =\frac{(1-x^2)^\lambda}{2\lambda}[1+\mc O(\langle\omega\rangle^{-1})]
 \end{align*}
 and the first representation follows.

 For the second representation, we need to exploit the fact that
  $u_0(y,0)=0$ to git rid of the apparent singularity of the Green
  function at $\lambda=0$. By the fundamental theorem of calculus we
  obtain
  \begin{align*}
    u_0(y,\lambda)
    &=\int_0^1 \partial_t u_0(y,t\lambda)dt
    =\lambda \int_0^1 (1+y)^{-t\lambda}\mc O((1-y)^0\langle
      t\omega\rangle^0)dt \\
      &\quad +\lambda\langle\log(1-y)\rangle\int_0^1 (1-y)^{-t\lambda}\mc O((1-y)^0\langle
      t\omega\rangle^0)dt.
  \end{align*}
  Inserting this expression for $u_0$ in the definition of the Green
  function yields the second representation.
\end{proof}

In order to estimate the kernel of the operators $T_\epsilon$ and
$\dot T_\epsilon$, we make frequent use of the following elementary bound.

\begin{lemma}
  \label{lem:a-b}
  We have $\langle a-b\rangle\gtrsim \langle a\rangle^{-1}\langle
  b\rangle$ for all $a,b\in \R$.
\end{lemma}

\begin{proof}
  If $|b|\leq 2|a|$ we have
  $\langle a\rangle^{-1}\langle b\rangle\lesssim 1\lesssim \langle
  a-b\rangle$ and if $|b|\geq 2|a|$,
  \[ \langle a-b\rangle\simeq 1+|a-b|\geq 1+|b|-|a|\geq
    1+\tfrac12|b|\simeq \langle b\rangle\gtrsim \langle
    a\rangle^{-1}\langle b\rangle. \]
\end{proof}

\begin{proposition}
  \label{prop:K}
 There exists an $\epsilon_0>0$ such that
  \[ K_\epsilon(s,y,x):=\frac{1}{2\pi
      i}\lim_{N\to\infty}\int_{\epsilon-iN}^{\epsilon+iN}
   e^{\lambda s} \left [G_V(y,x,\lambda)-G_0(y,x,\lambda)\right
   ]d\lambda \]
 exists for any $(s,y,x)\in [0,\infty)\times [0,1)\times [0,1)$ and
 $\epsilon\in (0,\epsilon_0]$ and we
 have
 \[ T_\epsilon(s) f(y)=\int_0^1 K_\epsilon(s,y,x)f(x)dx. \]
 Furthermore,
 \[ |K_\epsilon(s,y,x)|\lesssim e^{\epsilon s}\langle \log(1-y)\rangle^3\langle
   s+\log(1-x)\rangle^{-2} \]
 for all $(s,y,x)\in [0,\infty)\times [0,1)\times [0,1)$ and all
 $\epsilon\in (0,\epsilon_0]$.
\end{proposition}

\begin{proof}
 Lemma \ref{lem:GV-G0} yields the rough bound
\[ |G_V(y,x,\lambda)-G_0(y,x,\lambda)|\lesssim
  (1-y)^{-\frac14}(1-x)^{-\frac14}\langle\lambda\rangle^{-2} \]
for all $x,y\in [0,1)$ and $\lambda\in \C$ with $\Re\lambda\in
(0,\epsilon_0]$, provided $\epsilon_0>0$ is sufficiently small.
As a consequence, the existence of $K_\epsilon(s,y,x)$ follows and
Fubini's theorem yields the stated expression for $T_\epsilon(s)f$.

 To prove the bound on $K_\epsilon$, we need to distinguish between $\lambda$
 small and $\lambda$ large. To this end, we use a standard cut-off
 $\chi: \R\to [0,1]$ that satisfies $\chi(t)=1$ if
 $|t|\leq 1$ and $\chi(t)=0$ if $|t|\geq 2$.
 Then we split
 \begin{align*}
   K_\epsilon(s,y,x)
   &=\frac{e^{\epsilon s}}{2\pi}\int_\R e^{i\omega
     s}\left
   [G_V(y,x,\epsilon+i\omega)-G_0(y,x,\epsilon+i\omega)\right ]d\omega
   \\
   &=:\frac{e^{\epsilon s}}{2\pi}\left [I_\epsilon(s,y,x)+J_\epsilon(s,y,x)\right]
 \end{align*}
 and use Lemma \ref{lem:GV-G0} to decompose
 \begin{align*} I_\epsilon(s,y,x)
   &=\int_\R \chi(\omega)e^{i\omega s}\left
   [G_V(y,x,\epsilon+i\omega)-G_0(y,x,\epsilon+i\omega)\right
     ]d\omega \\
   &=\sum_{j=1}^4\underbrace{\int_\R \chi(\omega)e^{i\omega s}
     \widetilde G_{V,j}(y,x,\epsilon+i\omega)d\omega}_{=:I_{\epsilon,j}(s,y,x)}
   \end{align*}
and
\begin{align*}
  J_\epsilon(s,y,x)&=\int_\R [1-\chi(\omega)]e^{i\omega s}\left
   [G_V(y,x,\epsilon+i\omega)-G_0(y,x,\epsilon+i\omega)\right
                     ]d\omega \\
  &=\sum_{j=0}^4 \underbrace{\int_\R [1-\chi(\omega)]e^{i\omega
    s}G_{V,j}(y,x,\epsilon+i\omega)d\omega}_{=:J_{\epsilon,j}(s,y,x)}.
\end{align*}
By Lemma \ref{lem:GV-G0} we have
\begin{align*}
[1-\chi(\omega)]G_{V,1}(y,x,\epsilon+i\omega) 
&=[1-\chi(\omega)](1+y)^{-\epsilon}(1-x)^\epsilon(1+y)^{-i\omega}(1-x)^{i\omega}
  \\
&\quad \times (\epsilon+i\omega)^{-1}\mc
       O((1-y)^0(1-x)^0\langle\omega\rangle^{-1}) \\
&=e^{i\omega(-\log(1+y)+\log(1-x))}\mc O((1-y)^0(1-x)^0\langle\omega\rangle^{-2})
\end{align*}
and thus,
\begin{align*}
  |J_{\epsilon,1}(s,y,x)|
&\lesssim \left |\int_\R e^{i\omega(s-\log(1+y)+\log(1-x))}
\mc O((1-y)^0(1-x)^0\langle\omega\rangle^{-2})d\omega\right
                            | \\
&\lesssim \langle s-\log(1+y)+\log(1-x)\rangle^{-2} \\
&\lesssim \langle s+\log(1-x)\rangle^{-2}
\end{align*}
by means of two integrations by parts.
For $J_{\epsilon,2}$ we note that
\begin{align*}
  [1-\chi(\omega)]G_{V,2}(y,x,\epsilon+i\omega)&=1_{[0,1]}(x-y)[1-\chi(\omega)]
                                                 (1-y)^{-\epsilon}(1-x)^\epsilon
  (1-y)^{-i\omega}(1-x)^{i\omega} \\
  &\quad \times(\epsilon+i\omega)^{-1}\mc
    O((1-y)^0(1-x)^0\langle\omega\rangle^{-1}) \\
  &=e^{i\omega(-\log(1-y)+\log(1-x))}
\mc
    O((1-y)^0(1-x)^0\langle\omega\rangle^{-2}) 
  \end{align*}
  and we obtain
  \[ |J_{\epsilon,2}(s,y,x)|\lesssim \langle
    s-\log(1-y)+\log(1-x)\rangle^{-2}
    \lesssim \langle\log(1-y)\rangle^2 \langle
    s+\log(1-x)\rangle^{-2}.\]
  The terms $J_{\epsilon,3}$ and $J_{\epsilon,4}$ are handled
  analogously and in summary, we obtain
  \[ |J_{\epsilon}(s,y,x)|\lesssim \langle\log(1-y)\rangle^2\langle
    s+\log(1-x)\rangle^{-2}. \]

Now we turn to the low-frequency part $I_\epsilon$. By Lemma
\ref{lem:GV-G0} we have
\begin{align*}
  \chi(\omega)\widetilde G_{V,1}(y,x,\epsilon+i\omega)
  &=\int_0^1 (1+y)^{-it\omega}(1-x)^{i\omega}\mc
    O((1-y)^0(1-x)^0\langle\omega\rangle^{-2} t^0)dt \\
  &=\int_0^1 e^{i\omega(-t\log(1+y)+\log(1-x))}\mc O((1-y)^0(1-x)^0\langle\omega\rangle^{-2} t^0)dt
\end{align*}
and thus, by Fubini and two integrations by parts,
\begin{align*}
 | I_{\epsilon,1}(s,y,x)|
  &\leq\int_0^1 \left |\int_\R e^{i\omega(s-t\log(1+y)+\log(1-x))}\mc
    O((1-y)^0(1-x)^0\langle\omega\rangle^{-2}
    t^0)d\omega\right | dt \\
  &\lesssim \int_0^1 \langle s-t\log(1+y)+\log(1-x)\rangle^{-2} dt \\
    &\lesssim \int_0^1 \langle t\log(1+y)\rangle^2\langle s+\log(1-x)\rangle^{-2}dt \\
  &\lesssim \langle s+\log(1-x)\rangle^{-2}.
\end{align*}
For $\widetilde G_{V,2}$ we have
\[ \chi(\omega)\widetilde G_{V,2}(y,x,\epsilon+i\omega)
  =\int_0^1 \langle\log(1-y)\rangle
  (1-y)^{-it\omega}(1-x)^{i\omega}\mc
  O((1-y)^0(1-x)^0\langle\omega\rangle^{-2} t^0)dt \]
and thus,
\begin{align*}
  \frac{|I_{\epsilon,2}(s,y,x)|}{\langle\log(1-y)\rangle}
  &\leq \int_0^1 \left |\int_\R e^{i\omega(s-t\log(1-y)+\log(1-x))}
    \mc O((1-y)^0(1-x)^0\langle\omega\rangle^{-2}
    t^0)d\omega\right |dt \\
  &\lesssim \int_0^1 \langle s-t\log(1-y)+\log(1-x)\rangle^{-2}dt \\
  &\lesssim \int_0^1 \langle t\log(1-y)\rangle^2\langle
    s+\log(1-x)\rangle^{-2}dt \\
  &\lesssim \langle\log(1-y)\rangle^2\langle s+\log(1-x)\rangle^{-2}.
\end{align*}
The corresponding bound for $I_{\epsilon,3}$ follows analogously and
for $I_{\epsilon,4}$ we note that
\begin{align*}
  \chi(\omega)\widetilde G_{V,4}(y,x,\epsilon+i\omega)
  &=\int_0^1
    1_{[0,1]}(y-x)\langle\log(1-x)\rangle(1+y)^{-i\omega}(1+x)^{i\omega}
    (1-x)^{i(1-t)\omega} \\
  &\quad\times \mc
    O((1-y)^0(1-x)^0\langle\omega\rangle^{-2}t^0)dt \\
  &=\int_0^1 1_{[0,1]}(y-x)\langle\log(1-y)\rangle
    e^{i\omega(-\log(1+y)+\log(1+x)+(1-t)\log(1-x))} \\
  &\quad\times
    \mc O((1-y)^0(1-x)^0\langle\omega\rangle^{-2}t^0)dt 
\end{align*}
and thus,
\begin{align*}
  \frac{|I_{\epsilon,4}(s,y,x)|}{\langle\log(1-y)\rangle}
  &\leq \int_0^1 \Big |\int_\R 1_{[0,1]}(y-x)e^{i\omega(s-\log(1+y)+\log(1+x)+(1-t)\log(1-x))} \\
  &\quad\times
    \mc O((1-y)^0(1-x)^0\langle\omega\rangle^{-2}t^0)d\omega
    \Big |
    dt  \\
  &\lesssim \int_0^1 1_{[0,1]}(y-x)\langle
    s-\log(1+y)+\log(1+x)+(1-t)\log(1-x)\rangle^{-2}dt \\
  &\lesssim \int_0^1 1_{[0,1]}(y-x)\langle
    -\log(1+y)+\log(1+x)-t\log(1-x)\rangle^2 \\
  &\quad\times \langle s+\log(1-x)\rangle^{-2} dt \\
  &\lesssim \langle\log(1-y)\rangle^2 \langle s+\log(1-x)\rangle^{-2}.
\end{align*}
\end{proof}

Now we can conclude the desired bound for the operator $T_\epsilon$.

\begin{lemma}
  \label{lem:T}
  Let $p\in [2,\infty]$ and $q\in [1,\infty)$. 
  Then there exists an $\epsilon_0>0$ such that
  \[ \|e^{-\epsilon s}T_\epsilon(s)f\|_{L^p_s(0,\infty)L^q(0,1)}\lesssim \left
      \|(1-|\cdot|^2)^\frac12 f\right \|_{L^2(0,1)} \]
  for all $f\in C([0,1])$ and $\epsilon\in (0,\epsilon_0]$.
\end{lemma}

\begin{proof}
  By Proposition \ref{prop:K} we have
  \begin{align*}
    |e^{-\epsilon s}T_\epsilon(s)f(y)|
    &\leq \int_0^1 e^{-\epsilon s}|K_\epsilon(s,y,x)||f(x)|dx \\
    &=\int_0^\infty e^{-\epsilon
      s}|K_\epsilon(s,y,1-e^{-\eta})||f(1-e^{-\eta})|e^{-\eta}d\eta \\
    &\lesssim \langle\log(1-y)\rangle^3 \int_\R \langle
      s-\eta\rangle^{-2}1_{[0,\infty)}(\eta)
      |f(1-e^{-\eta})|e^{-\eta}d\eta
  \end{align*}
  and thus,
  \begin{align*}
    \|e^{-\epsilon s}T_\epsilon(s)f\|_{L^q(0,1)}\lesssim \int_\R
    \langle s-\eta\rangle^{-2}|1_{[0,\infty)}(\eta)f(1-e^{-\eta})| e^{-\eta}d\eta.
  \end{align*}
Consequently, Young's inequality yields
  \begin{align*}
    \|e^{-\epsilon s}T_\epsilon(s)f\|_{L^2_s(0,\infty)L^q(0,1)}^2
    &\lesssim \|\langle\cdot\rangle^{-2}\|_{L^1(\R)}^2 \int_0^\infty
      |f(1-e^{-\eta})|^2 e^{-2\eta}d\eta \\
    &\lesssim \int_0^1 (1-x)|f(x)|^2 dx.
  \end{align*}
  On the other hand, by Cauchy-Schwarz, we also have
  \[ \|e^{-\epsilon s}T_\epsilon(s)f\|_{L^q(0,1)}^2\lesssim \int_0^1
    (1-x)|f(x)|^2 dx \]
  for all $s\geq 0$ and this yields
  \[ \|e^{-\epsilon
      s}T_\epsilon(s)f\|_{L^\infty_s(0,\infty)L^q(0,1)}\lesssim \left
      \|(1-|\cdot|^2)^\frac12 f\right \|_{L^2(0,1)}. \]
\end{proof}

\subsection{Analysis of the operator $\dot T_\epsilon$}

The treatment of the operator $\dot T_\epsilon$ is very similar.

\begin{proposition}
  \label{prop:Tdot}
Let $p\in [2,\infty]$ and $q\in [1,\infty)$. Then there exists an
$\epsilon_0>0$ such that
\begin{align*}
  \|e^{-\epsilon s}\dot T_\epsilon(s)f\|_{L^p_s(0,\infty)L^q(0,1)}
  \lesssim \left \|(1-|\cdot|^2)^\frac12 f'\right\|_{L^2(0,1)}+\|f\|_{L^2(0,1)}
\end{align*}
for all $f\in C^1([0,1])$ and $\epsilon\in (0,\epsilon_0]$.
\end{proposition}

\begin{proof}
Let $\widetilde \chi: \R^2\to [0,1]$ be a smooth cut-off that
satisfies $\widetilde\chi(x)=1$ if $|x|\leq 1$ and $\widetilde
\chi(x)=0$ if $|x|\geq
2$. We define $\chi: \C\to [0,1]$ by $\chi(z):=\widetilde\chi(\Re
z,\Im z)$.
We split $\dot T_\epsilon(s)=\dot T_\epsilon^\flat(s)+\dot
T_\epsilon^\sharp(s)$, where
\begin{align*}
  \dot T_\epsilon^\flat(s)f(y)
  &:=\frac{1}{2\pi i}\lim_{N\to\infty}
  \int_{\epsilon-iN}^{\epsilon+iN}\chi(\lambda)\lambda e^{\lambda s}
    [G_V(y,x,\lambda)-G_0(y,x,\lambda)]f(x)dx d\lambda \\
    \dot T_\epsilon^\sharp(s)f(y)
  &:=\frac{1}{2\pi i}\lim_{N\to\infty}
  \int_{\epsilon-iN}^{\epsilon+iN}[1-\chi(\lambda)]\lambda e^{\lambda s}
    [G_V(y,x,\lambda)-G_0(y,x,\lambda)]f(x)dx d\lambda.
\end{align*}
For the low-frequency part $\dot T_\epsilon^\flat(s)f$, the additional
factor of $\lambda$ (compared to $T_\epsilon(s)f$) is helpful as it cancels the
singularity of the Green function at $\lambda=0$. Consequently, we
immediately infer the bound
\[ \|e^{-\epsilon s}\dot T_\epsilon^\flat(s) f\|_{L^p_s(0,\infty)L^q(0,1)}\lesssim \left
    \|(1-|\cdot|^2)^\frac12 f\right \|_{L^2(0,1)}\lesssim \|f\|_{L^2(0,1)} \]
by proceeding as in the proofs of Proposition \ref{prop:K} and Lemma
\ref{lem:T}.

The high-frequency part $\dot T_\epsilon^\sharp(s)f$ is more delicate
as the additional $\lambda$ destroys the inverse square decay of the
Green function as $|\Im\lambda|\to\infty$. Consequently, we need to
perform an integration by parts with respect to $x$ in order to
recover the decay. More precisely, by Lemma \ref{lem:GV-G0}, we have
\[ \dot T_\epsilon^\sharp(s)f(y)=\frac{1}{2\pi
    i}\lim_{N\to\infty}\sum_{j=1}^4
  \int_{\epsilon-iN}^{\epsilon+iN} [1-\chi(\lambda)]e^{\lambda
    s}\int_0^1 \lambda G_{V,j}(y,x,\lambda)f(x)dx d\lambda \]
and an integration by parts yields
\begin{align*}
  \int_0^1 &\lambda G_{V,1}(y,x,\lambda)f(x)dx \\
  &=(1+y)^{-\lambda}\int_y^1 (1-x)^\lambda \mc
    O((1-y)^0(1-x)^0\langle\omega\rangle^{-1})f(x)dx \\
  &=(1+y)^{-\lambda}(1-y)^\lambda\mc
    O((1-y)\langle\omega\rangle^{-2})f(y) \\
    &\quad
 +(1+y)^{-\lambda}\int_y^1 (1-x)^\lambda
    \mc O((1-y)^0(1-x)^0\langle\omega\rangle^{-2})f(x)dx \\
  &\quad +(1+y)^{-\lambda}\int_y^1 (1-x)^\lambda \mc
    O((1-y)^0(1-x)\langle\omega\rangle^{-2})f'(x)dx \\
           &=:(1+y)^{-\lambda}(1-y)^\lambda
             \mc O((1-y)\langle\omega\rangle^{-2})f(y)+\int_0^1
             H_{V,1}(y,x,\lambda)f(x)dx \\
           &\quad +
    \int_0^1 H'_{V,1}(y,x,\lambda)f'(x)dx.
\end{align*}
Note that the kernels $H_{V,1}$ and $H_{V,1}'$ are of the same type as
$G_{V,1}$.
By the same procedure we obtain an analogous representation of
$\int_0^1 \lambda G_{V,2}(y,x,\lambda)f(x)dx$.
The remaining two contributions produce an additional boundary term,
i.e., 
\begin{align*}
  \int_0^1 &\lambda G_{V,3}(y,x,\lambda)f(x)dx \\
  &=(1+y)^{-\lambda}\int_0^y (1-x)^\lambda \mc
    O((1-y)^0(1-x)^0\langle\omega\rangle^{-1})f(x)dx \\
           &=(1+y)^{-\lambda}\mc O(\langle\omega\rangle^{-2})f(0)
             +(1+y)^{-\lambda}(1-y)^\lambda \mc
    O((1-y)\langle\omega\rangle^{-2})f(y) \\
    &\quad +\int_0^1 H_{V,3}(y,x,\lambda)f(x)dx+\int_0^1 H'_{V,3}(y,x,\lambda)f'(x)dx
\end{align*}
and analogously for $G_{V,4}$. This means that
\begin{align*}
  e^{-\epsilon s}\dot T_\epsilon^\sharp(s)f(y)&=f(0)\int_\R e^{i\omega
    (s-\log(1+y))}\mc O(\langle \omega\rangle^{-2})d\omega \\
  &\quad +f(y)\int_\R e^{i\omega (s-\log(1+y)+\log(1-y))}\mc
    O((1-y)\langle\omega\rangle^{-2})d\omega \\
  &\quad +f(y)\int_\R e^{i\omega s}\mc
    O((1-y)\langle\omega\rangle^{-2})d\omega 
 +[A_\epsilon(s)f](y)+[B_\epsilon(s)f'](y) \\
  &=O(\langle s\rangle^{-1})f(0)+O(\langle
    s\rangle^{-1}(1-y)^\frac34)f(y)
    +[A_\epsilon(s)f](y)+[B_\epsilon(s)f'](y),
\end{align*}
where the operators $A_\epsilon(s)$ and $B_\epsilon(s)$ satisfy the
bound for $T_\epsilon(s)$ from Lemma \ref{lem:T}.
Consequently, we find
\begin{align*}
  \|e^{-\epsilon s}\dot T_\epsilon^\sharp(s)f\|_{L^p_s(0,\infty)L^q(0,1)}&\lesssim
  |f(0)|+\left
  \|(1-|\cdot|)^\frac34f\right\|_{L^q(0,1)} \\
  &\quad +\left \|(1-|\cdot|^2)^\frac12 f\right\|_{L^2(0,1)}+\left \|
    (1-|\cdot|^2)^\frac12 f' \right \|_{L^2(0,1)} \\
  &\lesssim \left \|(1-|\cdot|)^\frac34f\right\|_{L^\infty(0,1)}+\left
    \|(1-|\cdot|^2)^\frac12 f'\right\|_{L^2(0,1)}
    +\|f\|_{L^2(0,1)}
\end{align*}
and the simple estimate
\begin{align*}
  (1-y)^\frac34|f(y)|
  &\lesssim \int_y^1 (1-x)^\frac34|f'(x)|dx
  +\int_y^1 (1-x)^{-\frac14}|f(x)|dx \\
&\lesssim \left\|(1-|\cdot|^2)^\frac12
f'\right\|_{L^2(0,1)}+\|f\|_{L^2(0,1)}
\end{align*}                                          
for all $y\in [0,1]$ finishes the proof.
\end{proof}

\bibliographystyle{plain} \bibliography{../refs/refs}

  \end{document}